\documentclass{llncs}
   
      \usepackage[T1]{fontenc}
      \usepackage{color}
     \usepackage[english]{babel}

\usepackage{shuffle}
\usepackage{etex}
\usepackage{proofCF}
\usepackage{tikz}
\usepackage[toc,page]{appendix}
\usepackage{amsfonts, amsmath, amssymb, latexsym,amstext}
\usepackage[all]{xy}
\usepackage{textcomp}
\usepackage{mathabx}
\usepackage{newlfont}
\usepackage{bussproofs,xargs,ifthen,xstring}
\newdimen\proofrulebreadth \proofrulebreadth=.05em
\newdimen\proofdotseparation \proofdotseparation=1.25ex
\newdimen\proofrulebaseline \proofrulebaseline=2ex
\newcount\proofdotnumber \proofdotnumber=3
\let\then\relax
\def\hfi{\hskip0pt plus.0001fil}
\mathchardef\squigto="3A3B
\newif\ifinsideprooftree\insideprooftreefalse
\newif\ifonleftofproofrule\onleftofproofrulefalse
\newif\ifproofdots\proofdotsfalse
\newif\ifdoubleproof\doubleprooffalse
\let\wereinproofbit\relax
\newdimen\shortenproofleft
\newdimen\shortenproofright
\newdimen\proofbelowshift
\newbox\proofabove
\newbox\proofbelow
\newbox\proofrulename
\def\shiftproofbelow{\let\next\relax\afterassignment\setshiftproofbelow\dimen0 }
\def\shiftproofbelowneg{\def\next{\multiply\dimen0 by-1 }%
\afterassignment\setshiftproofbelow\dimen0 }
\def\setshiftproofbelow{\next\proofbelowshift=\dimen0 }
\def\setproofrulebreadth{\proofrulebreadth}

\def\prooftree{
\ifnum	\lastpenalty=1
\then	\unpenalty
\else	\onleftofproofrulefalse
\fi
\ifonleftofproofrule
\else	\ifinsideprooftree
	\then	\hskip.5em plus1fil
	\fi
\fi
\bgroup
\setbox\proofbelow=\hbox{}\setbox\proofrulename=\hbox{}%
\let\justifies\proofover\let\leadsto\proofoverdots\let\Justifies\proofoverdbl
\let\using\proofusing\let\[\prooftree
\ifinsideprooftree\let\]\endprooftree\fi
\proofdotsfalse\doubleprooffalse
\let\thickness\setproofrulebreadth
\let\shiftright\shiftproofbelow \let\shift\shiftproofbelow
\let\shiftleft\shiftproofbelowneg
\let\ifwasinsideprooftree\ifinsideprooftree
\insideprooftreetrue
\setbox\proofabove=\hbox\bgroup$\displaystyle 
\let\wereinproofbit\prooftree
\shortenproofleft=0pt \shortenproofright=0pt \proofbelowshift=0pt
\onleftofproofruletrue\penalty1
}

\def\eproofbit{
\ifx	\wereinproofbit\prooftree
\then	\ifcase	\lastpenalty
	\then	\shortenproofright=0pt	
	\or	\unpenalty\hfil		
	\or	\unpenalty\unskip	
	\else	\shortenproofright=0pt	
	\fi
\fi
\global\dimen0=\shortenproofleft
\global\dimen1=\shortenproofright
\global\dimen2=\proofrulebreadth
\global\dimen3=\proofbelowshift
\global\dimen4=\proofdotseparation
\global\count10=\proofdotnumber
$\egroup  
\shortenproofleft=\dimen0
\shortenproofright=\dimen1
\proofrulebreadth=\dimen2
\proofbelowshift=\dimen3
\proofdotseparation=\dimen4
\proofdotnumber=\count10
}

\def\proofover{
\eproofbit 
\setbox\proofbelow=\hbox\bgroup 
\let\wereinproofbit\proofover
$\displaystyle
}%
\def\proofoverdbl{
\eproofbit 
\doubleprooftrue
\setbox\proofbelow=\hbox\bgroup 
\let\wereinproofbit\proofoverdbl
$\displaystyle
}%
\def\proofoverdots{
\eproofbit 
\proofdotstrue
\setbox\proofbelow=\hbox\bgroup 
\let\wereinproofbit\proofoverdots
$\displaystyle
}%
\def\proofusing{
\eproofbit 
\setbox\proofrulename=\hbox\bgroup 
\let\wereinproofbit\proofusing
\kern0.3em$
}

\def\endprooftree{
\eproofbit 
  \dimen5 =0pt
\dimen0=\wd\proofabove \advance\dimen0-\shortenproofleft
\advance\dimen0-\shortenproofright
\dimen1=.5\dimen0 \advance\dimen1-.5\wd\proofbelow
\dimen4=\dimen1
\advance\dimen1\proofbelowshift \advance\dimen4-\proofbelowshift
\ifdim	\dimen1<0pt
\then	\advance\shortenproofleft\dimen1
	\advance\dimen0-\dimen1
	\dimen1=0pt
	\ifdim  \shortenproofleft<0pt
        \then   \setbox\proofabove=\hbox{%
			\kern-\shortenproofleft\unhbox\proofabove}%
                \shortenproofleft=0pt
        \fi
\fi
\ifdim	\dimen4<0pt
\then	\advance\shortenproofright\dimen4
	\advance\dimen0-\dimen4
	\dimen4=0pt
\fi
\ifdim	\shortenproofright<\wd\proofrulename
\then	\shortenproofright=\wd\proofrulename
\fi
\dimen2=\shortenproofleft \advance\dimen2 by\dimen1
\dimen3=\shortenproofright\advance\dimen3 by\dimen4
\ifproofdots
\then
	\dimen6=\shortenproofleft \advance\dimen6 .5\dimen0
	\setbox1=\vbox to\proofdotseparation{\vss\hbox{$\cdot$}\vss}%
	\setbox0=\hbox{%
		\advance\dimen6-.5\wd1
		\kern\dimen6
		$\vcenter to\proofdotnumber\proofdotseparation
			{\leaders\box1\vfill}$%
		\unhbox\proofrulename}%
\else	\dimen6=\fontdimen22\the\textfont2 
	\dimen7=\dimen6
	\advance\dimen6by.5\proofrulebreadth
	\advance\dimen7by-.5\proofrulebreadth
	\setbox0=\hbox{%
		\kern\shortenproofleft
		\ifdoubleproof
		\then	\hbox to\dimen0{%
			$\mathsurround0pt\mathord=\mkern-6mu%
			\cleaders\hbox{$\mkern-2mu=\mkern-2mu$}\hfill
			\mkern-6mu\mathord=$}%
		\else	\vrule height\dimen6 depth-\dimen7 width\dimen0
		\fi
		\unhbox\proofrulename}%
	\ht0=\dimen6 \dp0=-\dimen7
\fi
\let\doll\relax
\ifwasinsideprooftree
\then	\let\VBOX\vbox
\else	\ifmmode\else$\let\doll=$\fi
	\let\VBOX\vcenter
\fi
\VBOX	{\baselineskip\proofrulebaseline \lineskip.2ex
	\expandafter\lineskiplimit\ifproofdots0ex\else-0.6ex\fi
	\hbox	spread\dimen5	{\hfi\unhbox\proofabove\hfi}%
	\hbox{\box0}%
	\hbox	{\kern\dimen2 \box\proofbelow}}\doll%
\global\dimen2=\dimen2
\global\dimen3=\dimen3
\egroup 
\ifonleftofproofrule
\then	\shortenproofleft=\dimen2
\fi
\shortenproofright=\dimen3
\onleftofproofrulefalse
\ifinsideprooftree
\then	\hskip.5em plus 1fil \penalty2
\fi
}

\usepackage[english]{babel}
\usepackage{pifont} 
\usepackage{verbatim}  
\usepackage{cmll}
\usepackage{ mathdots }
\usepackage{ tipa }
\usepackage{seqproofs}
\usepackage{color}
\usepackage{stmaryrd}

\newcommand\seq\vdash 
   
\newcommand\tms\otimes
\newcommand\ra\rightarrow
\newcommand\Ra\Rightarrow
\newcommand\llts\otimes
\newcommand\pt\bullet 
\newcommand\lts\bullet

 \newcommand\dw\downarrow

\newtheorem{defi}{ \bf Definition}

\newtheorem{prop}{{\bf Proposition}}

\newtheorem{lemme}{{\bf Lemma}}
\newtheorem{rem}{{\bf Remark}}
\newcommand{\bdm}{\begin{displaymath}}
\newcommand{\edm}{\end{displaymath}}
\newcommand{\de}{\mathfrak}
\newcommand{\da}{\scriptstyle\maltese}
\newcommand{\N}{ \mathbb{N}at}
\newcommand{\view}[1]{\raisebox{.3ex}{$\ulcorner$}{#1}\raisebox{.3ex}{$\urcorner$}}

\newcommand{\Ln}{\mathbb{L}_{n}}
\newcommand{\la}{\lambda}
\newcommand{\A}{\mathbb{A}}
\newcommand{\B}{\mathbb{B}}

\newcommand{\R}{\Rightarrow}

\newcommand{\Sig}{(\Sigma x \in \A) \B(x)}
\newcommand{\fr}{\A \R \B}

\newlength{\Widetildeheight}
\newlength{\Widetildewidth}

\newenvironment{varitemize}
{
\begin{list}{\labelitemi}
{\setlength{\itemsep}{0pt}
 \setlength{\topsep}{0pt}
 \setlength{\parsep}{0pt}
 \setlength{\partopsep}{0pt}
 \setlength{\leftmargin}{15pt}
 \setlength{\rightmargin}{0pt}
 \setlength{\itemindent}{0pt}
 \setlength{\labelsep}{5pt}
 \setlength{\labelwidth}{10pt}
}}
{
 \end{list}
}

\begin{document}
     \title{Type Theory in Ludics}
    \author{E.Sironi}
     
\institute{Aix Marseille Universit\'e, CNRS, Centrale Marseille, I2M, \\
UMR 7373, 13453 Marseille, France
\\eugenia.sironi@univ-mrs.fr}
  \maketitle

\begin{abstract}    
We present some first steps in the more general setting of the interpretation of dependent type theory in Ludics. The framework is the following: a (Martin-L\"of) type A is represented by a behaviour (which corresponds to a formula) in such a way that canonical elements of A are interpreted in a set that is principal for the behaviour, where principal means in some way a minimal generator. We introduce some notions on Ludics and the interpretation of Martin-L\"of rules. Then we propose a representation for simple types in Ludics, i.e., natural numbers, lists, the arrow construction and the usual constructors.
\end{abstract}

     \section{Introduction}
The aim of this paper is to present some first steps in the more general setting of the interpretation of dependent type theory in Ludics.\\
Dependent type theories started in the early 1970's, when Martin-L\"of introduced his intuitionistic theory of types \cite{MartinLF}. Types have been studied from the initial motivation to improve the paradoxical structure of sets, and were found to be much closer to the notion of computation, thanks to the Curry-Howard isomorphism. This latter is a one-to-one correspondence between logical systems and type systems such that propositions are mapped to types and proofs of a proposition are mapped to terms of the type (which is the image of the proposition).
\\Ludics is a theory introduced by Girard in \cite{LocSol} and comes from a fine analysis of the multiplicative, additive fragment of Linear Logic (MALL). The aim of Ludics is to overcome the distinction between syntax and semantics. These two worlds usually distinct become a unique universe, where an object is completely determined by the objects it interacts with. This means that properties are expressed and tested internally (and interactively), because the objects themselves test each other. The main object of Ludics is called \textbf{design} and represents a cut-free (para)-proof\footnotemark\footnotetext{Not a proof because there is a special rule, the daimon, which ends the interaction.} of a certain formula where everything is erased, but locations. With the word location we mean the ``place'' occupied by a subformula of a formula. The notion of location is based on an intuition given by computer science: proofs do not manipulate the ``idea'' of a formula, but the address in the memory where it is stored: its location \cite{FaggianTesi}. A design may be read as a representation of a formula through the addresses of its subformulae (and recursively). Designs represent both an abstraction of a formal proof and its semantic interpretation, therefore syntax and semantics meet in this notion. Ludics may also be represented in the spirit of $\la$-calculus: a design is a linear $\la$-term with ``superimposed abstractions'', interaction is similar to $\beta$-reduction \cite{TeruiComp}.
\\ A way to present Ludics is to see it as a form of game semantics, where designs can be seen as a linear version of innocent strategies \cite{FaggianH02}, \cite{CurFag}. However there are two main differences between Ludics and game semantics. In Ludics, the notion of interaction, i.e., normalization, comes at first from which we define interactive type\footnotemark\footnotetext{This notion is introduced by Terui in \cite{TeruiComp}.} (called \textbf{behaviour}): a set of designs which behave well with respect to interaction. Second, designs are a priori untyped: their type depends on the behaviour in which we consider them, indeed designs are characterized by the set of the other designs they well interact with. This feature of Ludics corresponds to the internal notion of proof which is one of the principal features of type theory. In Martin-L\"of type theory a type is characterized by the set of its terms, and in Ludics a behaviour (which corresponds to a formula) is characterized by a particular subset of its designs, i.e., the designs that generate it (called \textbf{material}). In \cite{MartinLF} Martin-L\"of introduces a constructive set theory. Ludics is even more constructive than \cite{MartinLF}, indeed designs are explicitly constructed. So Ludics looks like a good candidate to give a concrete model of type theory, in particular we show in this paper that Ludics can be a model for dependent types (types that depend on a value). At the same time the discussion of first order in Ludics is still an open question, so a representation of dependent types in Ludics could help to investigate about this subject. The framework is the following: a (Martin-L\"of) type $A$ is represented by a behaviour $\textbf{A}$ in such a way that canonical elements of $A$ are interpreted in a set $\mathbb{A}$ that is principal for $\textbf{A}$, where principal (defined below) means in some way a minimal generator.\\
In section \ref{MLLud} we introduce some notions on Ludics and Martin-L\"of Type Theory. In section \ref{SympleTyp} we propose a representation for simple types in Ludics, i.e., natural numbers, lists, the arrow construction $\to$ for them. In section \ref{DepTyp} we  propose a representation for constructors $\Pi$ and $\Sigma$ and an example of dependent type.\\Being Ludics affine\footnotemark\footnotetext{The argument of a function has to be used at most once. We have weakening but we do not have contraction.} we are not able to represent some basic operations on simple types, say the square function on $\mathbb{N}$. However there exist extensions of Ludics that integrate exponentials \cite{Ludicswithrep}, on which our approach may be applied.

\section{Interpreting Martin-L\"of's Type Theory in Ludics}\label{MLLud}

\subsection{Martin-L\"of's Type Theory}
Martin-L\"of defines a constructive set theory, where no external notion on sets has to be assumed. Types are defined by judgements, giving their meaning. Judgements are of the form $\Gamma \vdash A$ type, $\Gamma \vdash x=y : A$, $\Gamma \vdash t: A$, where $t$ is a term of type $A$ w.r.t. a context $\Gamma$. Four kinds of rules help to construct judgements: the \textbf{formation} rule which explains how to form a new type (eventually using other types already defined), the \textbf{introduction} rule which explains what is a canonical element of a given type, the \textbf{elimination} rule, i.e., how to define functions over the types defined with the introduction rules, the  \textbf{equality} rule, i.e., how to compute the functions defined by elimination over the canonical elements of a type defined with the introduction rule.

\begin{example} \label{esprod}
Given a type $A$ and a family of types $B(x)$ over $A$ we recall the rules for the type $(\Pi x\in A)B(x)$ defined in \cite{MartinLF} as follows:\\
-- $\Pi$-Formation:
\begin{minipage}{.46\textwidth}
$\shortstack{ $\Gamma \vdash A$ type  \hspace{1em} \shortstack{$\Gamma, x : A \vdash B(x)$ type} \\\hrulefill\\ $\Gamma \vdash  (\Pi x\in A)B(x)$ type}$ 
\end{minipage}
\begin{minipage}{.33\textwidth}
which states that the product of a family of types is a type.
\end{minipage}
\vspace{.1mm}
\\%
-- $\Pi$-Introduction: 
\begin{minipage}{.39\textwidth}
$\shortstack{$\Gamma, x : A \vdash b(x) : B(x)$ \\\hrulefill\\ $ \Gamma \vdash (\lambda x) b(x) : (\Pi x \in A)B(x)$}$
\end{minipage}
\begin{minipage}{.38\textwidth}
 which states that a canonical element of the product of a family of types is an abstraction $(\lambda x) b(x)$.
\end{minipage}
-- $\Pi$-Equality:
\\
\centerline{
\begin{tabular}{ccc}
$\shortstack{$\Gamma \vdash a : A$ \hspace{1em} \shortstack{ $\Gamma, x : A \vdash b(x) : B(x)$}\\\hrulefill\\ $\Gamma \vdash Ap((\lambda x)b(x),a)= b(a): B(a)$}$
&~~~~&
$\shortstack{$\Gamma \vdash c : (\Pi x\in A) B(x)$ \\\hrulefill\\ $\Gamma \vdash c=(\lambda x) Ap (c,x): (\Pi x\in A) B(x)$}$
\end{tabular}
}
\\
The first equality rule shows how the function $Ap$ operates on canonical elements of the product. The second equality rule says that $c$ and $(\lambda x)Ap(c,x)$ are equal: they yield the same canonical element of $(\Pi x \in A)B(x)$.
\end{example}

\subsection{A few Words about Ludics}
In this section we introduce some notions of Ludics, we refer the reader to \cite{LocSol} for a formal thorough presentation. Ludics is a theory introduced by Girard \cite{LocSol} to reconstruct logic starting from the notion of interaction. The central object is no more truth or proof, but interaction defined on designs. To define designs we first introduce the notions of address, action and chronicle. 

\begin{defi}
 An \textbf{address}\footnotemark\footnotetext{The addresses are denoted by greek letters: $\gamma, \xi,...$}  $\xi$ is a finite (maybe empty) sequence of integers.
\\An \textbf{action} $\kappa$ is 

\begin{varitemize}

\item either a positive proper action $(+, \xi, I)$ or a negative proper action $(-, \xi, I)$ where $\xi$ is called the \textbf{address} of $\kappa$ and the finite set of integers $I$ is said its \textbf{ramification}. 

\item or the positive (non proper) action \textbf{daimon} denoted by $\da$. 
\end{varitemize}
\end{defi}

\begin{defi}
A \textbf{chronicle} $\de{c}$ is a non empty, finite alternate sequence of actions s.t. (1) Each action of $\de{c}$ is either initial\footnotemark\footnotetext{First action of the chronicle if the action is negative.} or justified\footnotemark\footnotetext{An action $\kappa$ is justified by the action $\kappa'$ when the address of $\kappa$ is built from the address of $\kappa '$. For instance $(+, \xi.0.2, \{ 0\})$ is justified by $(-, \xi.0, \{2 \})$ and $(-, \alpha. 3, \{ 2\})$ is justified by $(+, \alpha, \{ 3\})$.} by a previous action of opposite polarity. In particular negative actions are justified by the immediately precedent positive action. 
(2) Actions have distinct addresses. 
(3) If present, a daimon is the last action of the chronicle. 
\end{defi}

\begin{defi}
 Two chronicles $\de{c}_{1}$ and $\de{c}_{2}$ are \textbf{coherent}, noted $\de{c}_1 \coh \de{c}_{2}$, when the two following conditions are satisfied: 

\begin{varitemize}

\item Comparability: Either one extends the other or they first differ on negative actions, i.e., if $w\kappa_{1} \coh w\kappa_{2}$ then either $\kappa_{1}=\kappa_{2}$ or $\kappa_{1}$ and $\kappa_{2}$ are negative actions. 

\item Propagation: When they first differ on negative actions and these negative actions have distinct addresses then the addresses of following actions in $\de{c}_{1}$ and $\de{c}_{2}$ are pairwise distinct, i.e., if  $w( -, \xi_{1} , I_{1} )w_{1} \kappa_{1} \coh  w(-, \xi_{2} , I_{2} )w_{2} \kappa_{2}$ with $\xi_{1} \neq \xi_{2}$ then $\kappa_{1}$ and $\kappa_{2}$ have distinct addresses. 
\end{varitemize}
\end{defi}
We consider chronicles based on a sequent $\Gamma \vdash \Delta$, where $\Delta$ is a finite set of addresses, $\Gamma$ contains at most one address and the addresses of $\Gamma \cup \Delta$ are pairwise disjoint, i.e., no address is a subaddress\footnotemark\footnotetext{$\xi$ is a subaddress of $\alpha$ when $\xi$ is a prefix of $\alpha$. For instance $1.0.2$ is a subaddress of $1.0.2.4.0$.} of another one. $\Delta$ contains the addresses of the initial positive actions of the chronicle. If $\Gamma$ is empty the \textbf{base} is said \textbf{positive}, otherwise the base is said \textbf{negative} and $\Gamma$ contains the address of the initial negative action.

\begin{defi} 
A \textbf{design} $\de{D}$, based on $\Gamma \vdash \Delta$, is a set of chronicles based on $\Gamma \vdash \Delta$ s.t. the following conditions are satisfied:

\begin{varitemize}

\item \emph{Forest}: The set is prefix closed. 

\item \emph{Coherence}: The set is a clique with respect to $\coh$.

\item \emph{Positivity}: A chronicle without extension in $\de{D}$ (also said \textbf{maximal}) ends with a positive action.

\item \emph{Totality}: $\de{D}$ is non empty when the base is positive, in that case all the chronicles begin with a (unique) positive action. 

\end{varitemize} 

\end{defi}
We say that a design is \textbf{positive}, when its base $\Gamma \vdash \Delta$ is positive ($\Gamma$ is empty), otherwise it is said \textbf{negative}. 
\\A design can also be represented as a proof-like sequent structure. 
\begin{example}\label{des}
In the left the design $\de{D}$ is represented as a tree-like structure of actions, while in the right as a proof-like sequent structure.
\begin{center}
\begin{minipage}{.48\textwidth}
\scalebox{.9}{
\begin{tikzpicture}[x=70pt,y=30pt]
\node at (0,0) {$(+,\xi,\{1,3\})$};
\node at (-1,.8) {$(-,\xi.3,\{0\})$};
\node at (-1,1.6) {$(+,\xi.3.0,\emptyset)$};
\node at (0,.8) {$(-,\xi.1,\{0\})$};
\node at (0,1.6) {$(+,\xi.1.0,\{0\})$};
\node at (1,.8) {$(-,\xi.1,\{1\})$};
\node at (1,1.6) {$(+,\xi.1.1,\{0\})$};

\draw (0,.3) -- (0,.6);
\draw (0,.3) -- (-1,.6);
\draw (0,.3) -- (1,.6);
\draw (-1,1) -- (-1,1.3);
\draw (0,1) -- (0,1.3);
\draw (1,1) -- (1,1.3);
\end{tikzpicture}
}
\end{minipage}
\begin{minipage}{.48\textwidth}
\begin{center}
\scalebox{.9}{
\infer{\vdash \xi}
	{
	\infer{\xi.3 \vdash}
		{
		\infer[\emptyset]{\vdash \xi.3.0}
			{}
		}
	&
	\infer{\xi1 \vdash}
		{
		\infer{\vdash \xi.1.0}
			{
			\xi1.0.0 \vdash
			}
		&
		\infer{\vdash \xi.1.1}
			{
			\xi.1.1.0 \vdash
			}
		}
	}

}
\end{center}
\end{minipage}
\end{center}

\end{example}
When we consider a design as a set of chronicles, we abusively only write maximal chronicles. For instance, in Example \ref{des}, $\de{D} =\{ (+, \xi, \{  1,3 \}) (-, \xi.3, \{ 0\}) (+, \xi.3.0, \emptyset),$ 
$ (+, \xi, \{  1,3\} ) (-, \xi.1, \{ 0\}) (+, \xi.1.0, \{ 0\} ),  (+, \xi, \{  1,3\} ) (-, \xi.1, \{ 1 \}) (+, \xi.1.1, \{ 0\} ) \}$. 
\\When we consider a design as a proof-like sequent structure we associate a \textbf{positive rule} to each positive action and a \textbf{negative rule} to all negative actions with the same address. For instance we associate the rule $(-, \xi.1, \{  \{0 \}, \{ 1\} \})$ to $(-, \xi.1,  \{0 \})$ and $(-, \xi.1,  \{1 \})$ and write it as \scalebox{.8}{$\shortstack{  \shortstack{$\vdash \xi.1.0$} \hspace{1em}  \shortstack{$\vdash \xi.1.1$}  \\\hrulefill\\ $\xi.1 \vdash $}$}. Given a negative base $ \sigma   \vdash \Gamma$,  the rule \scalebox{.8}{$\shortstack{   $\vdash \Gamma$  \\\hrulefill\\ $ \sigma   \vdash \Gamma$}$} represents $(-, \sigma, \emptyset)$. 

\begin{example}\label{dai}
 The design whose only action is $\da$ is called $\de{Dai}$.
 \\ $\de{Dai}=\{\da$$\}$ \hspace{2em} $ \de{Dai} \quad = \shortstack{$\hrulefill_{\da}$ \\ $\vdash \xi.1,...,\xi. n$}$.

 \end{example} 
The \textbf{base} of a design is its first sequent in the bottom-up view, i.e., its root, when we consider it like a tree. For instance in Example \ref{des} the base is $\vdash \xi$, while in Example \ref{dai} the base is $\vdash \xi.1,...,\xi.n$.

\begin{defi}
 A \textbf{cut} is an address which appears in the base of two designs with opposite polarity (on the left and on the right of $\vdash$).
 \\ A \textbf{net} is a finite set of designs. A \textbf{cut-net} is a net where 
   \begin{varitemize}
    \item all the addresses occurring in the bases are pairwise disjoint or equal,
    \item each address appears in at most two bases, in this case it is a cut, 
  \item the graph whose vertices are the addresses and whose edges are the cuts is connected and acyclic. 
    \end{varitemize}
Given a cut-net we can distinguish a particular design, called \textbf{main design}, 
it is the only positive design of the cut-net, if there is one. Otherwise it is the only negative design whose base contains an address that is not part of a cut. The first rule (in the bottom up view) of the main design is called the \textbf{main rule}.
\\A cut-net is \textbf{closed} when all addresses in bases are part of a cut. 
\\ We remark that in the case of a closed cut-net, the main design is a positive design, then its main action is positive. 

\end{defi}
Interaction, i.e., cut-elimination, is defined on cut-nets. First we consider the case of a closed cut-net, in this case if the interaction ends (without failing) the result is $\{ \da$$\}$, while in the general case it can be a design $\de{D} \neq \{ \da$$\}$.
\begin{defi}
Let $\de{R}$ be a closed cut-net. The design resulting from the \textbf{interaction}, denoted by $\llbracket \de{R}\rrbracket$ and called the \textbf{normalization} of $\de{R}$, is defined in the following way: let $\de{D}$ be the main design of $\de{R}$, with first action $\kappa$,
\begin{varitemize}

\item \textbf{Daimon}: if $\kappa$ is the daimon, then $\llbracket \de{R}\rrbracket = \{ \da$$\} $

\item otherwise $\kappa$ is a proper positive action $(+, \sigma, I)$ such that $\sigma$ is part of a cut with another design with last rule $(-, \sigma, N)$, ($N$ aggregates ramifications of actions on the same address $\sigma$)

\begin{varitemize}

\item \textbf{Failure}: If $I \notin N$, the interaction fails. 

\item \textbf{Conversion}: otherwise, the interaction follows the connected part of subdesigns\footnotemark\footnotetext{A subdesign of a design $\de{D}$ is a subtree of $\de{D}$, that is still a design.} obtained from $I$ with the rest of $\de{R}$. 

\end{varitemize}

\end{varitemize}

\end{defi}

\begin{defi}
Now let's consider \textbf{the general case}, where the net is not supposed to be closed. Thus the main rule can be positive or negative, and besides the cases of the precedent definition there are two new possibilities :
 \begin{varitemize}
 \item \textbf{Positive commutation}:  the net is positive, with main rule $(+,\xi, I)$ but  $\xi$ is not a cut. Let $\de{D_i}$ be as in the case of conversion above, and define $\de{R}'$ by replacing $\de{D}$ with the $\de{D_i}$. $\de{R}'$ splits into several connected components, and each $\de{D_i}$ lies in a component $\de{R}_i$, which is a net, and the $\de{R}_i$ are pairwise distinct. Let the $\de{E}_i$ be the respective normal forms of the $\de{R}_i$  (they exists because the $\de{R}_i$ are negative). The normal form of $\de{R}$ is the design whose first rule is $(+,\xi,I)$ and which proceeds with $\de{E}_i$ above the premise of index $i$.

  \item \textbf{Negative commutation}: The net is negative, with main design $\de{D}$ and main rule $(-,\xi,N)$. For $I \in N$ let $\de{D}_I$ be the subdesign of $\de{D}$ above the premise of index $I$ of the last rule, and let us replace $\de{D}$ with $\de{D}_I$ in $\de{R}$, and let $\de{R}_I$ be the connected component of $\de{D}_I$ (we don't directly get a net, as above, because of weakening). Let $N'$ be the subset of $N$ made of those $I$ for which $\de{R}_I$ has a normal form $\de{E}_I$. The normal form of $\de{R}$ is defined as the design ending with $(-,\xi,N)$ and which proceeds with $\de{E}_I$ above the premise of index I.
   \end{varitemize}

\end{defi}
In other terms, the positive commutation recopies the first rule (in the bottom-up view) and then proceeds separately above each premise. The negative commutation does the same, but some premises may disappear.

\begin{example}
Let $\de{E}, \de{F}$ be the following designs.\\  
\scalebox{.8}
{$\de{E}= \shortstack{ \shortstack{$\hrulefill_{\emptyset}$ \\$\vdash \alpha.0.0$ \\\hrulefill\\ $\alpha.0 \vdash $} \hspace{1em}  \shortstack{$\hrulefill_{\emptyset}$ \\$\vdash \alpha.1.1$ \\\hrulefill\\ $\alpha.1 \vdash $}  \\\hrulefill\\ $\vdash \alpha$}$, \hspace{2em} 
$\de{F}= \shortstack{$\vdots$\\$\de{G}$ \\$\vdash \beta$ \\\hrulefill\\  $\alpha.1.1 \vdash \beta$ \\\hrulefill\\ $\vdash \alpha.1, \beta$ \\\hrulefill\\ $\alpha.0.0 \vdash \alpha.1, \beta$ \\\hrulefill\\ $\vdash \alpha.0, \alpha.1, \beta$ \\\hrulefill\\ $ \alpha \vdash \beta $}$ \hspace{2em} 
$\llbracket \de{E}, \de{F}\rrbracket= \de{G}$
}. \\In therms of chronicles it corresponds to 

\scalebox{.8}
{
\begin{tikzpicture}[x=70pt,y=30pt]
\node at (-.6,0) {$\de{E}=$};
\node at (0,0) {$(+,\alpha,\{0,1\})$};
\node at (-1,1) {$(-,\alpha.0,\{0\})$};
\node at (-1,2) {$(+,\alpha.0.0,\emptyset)$};
\node at (1,1.5) {$(-,\alpha.1,\{1\})$};
\node at (1,2.5) {$(+,\alpha.1.1,\emptyset)$};

\draw(1,1.7)--(1,2.3);
\draw(-1,1.2)--(-1,1.8);
\draw(0,.2)--(-1,.8);
\draw(0,.2)--(1,1.3);

\node at (3.6,0) {$= \de{F}$};
\node at (3,0) {$(-, \alpha, \{0,1 \})$};
\node at (3,1) {$(+, \alpha.0, \{0\})$};
\node at (3,2) {$(-, \alpha.0.0, \emptyset)$};
\node at (3,3) {$(+, \alpha.1, \{1 \})$};
\node at (3,4) {$(-, \alpha.1.1, \emptyset)$};
\node at (3,4.5) {$\de{G}$};

\draw(3,.2)--(3,.8);
\draw(3,1.2)--(3,1.8);
\draw(3,2.2)--(3,2.8);
\draw(3,3.2)--(3,3.8);
\draw(3,4.2)--(3,4.3);

\draw[dashed,->,red] (.45,0) -- (2.45,0);
\draw[dashed,->,red] (2.45,0) -- (2.45,1);
\draw[dashed,->,red] (2.45,1) -- (-.45,1);
\draw[dashed,->,red] (-.45,1) -- (-.45,2);
\draw[dashed,->,red] (-.45,2) -- (2.45,2);
\draw[dashed,->,red] (2.45,2) -- (2.45,3);
\draw[dashed,->,red] (2.45,3) -- (1.45,1.5);
\draw[dashed,->,red] (1.45,1.5) -- (1.45,2.5);
\draw[dashed,->,red] (1.45,2.5) -- (2.45,4);
\draw[dashed,->,red] (2.45,4) -- (2.45,4.5);

\end{tikzpicture} 
} \\The dashed line represents the interaction between $\de{E}$ and $\de{F}$. 

\end{example}

\begin{defi}
A design $\de{E}$ and a net $\de{R}$ are \textbf{orthogonal}, noted $\de{E} \perp \de{R}$, when $ \llbracket \de{E}, \de{R} \rrbracket = \{ \da$$\}  $. A set $E$ of designs with the same base is called a \textbf{behaviour} when it is equal to its biorthogonal\footnotemark\footnotetext{ $E^{\perp}$ is the set of designs orthogonal to all the elements of $E$.}, i.e., $E= E^{\perp\perp}$.
\end{defi}


 \begin{defi}
 Given a design $\de{D}$ we define its \textbf{incarnation} in a behaviour $\textbf{G}$ as $|\de{D}|_{\textbf{G}}= \bigcap \{  \de{D}' \, | \, \de{D}' \subseteq \de{D}, \de{D}' \in \textbf{G} \}$.
\\$\de{D}$ is \textbf{material} in $\textbf{G}$ when it is equal to its incarnation in $\textbf{G}$, i.e., $\de{D}= |\de{D}|_{\textbf{G}}$.
\\The \textbf{incarnation} of $\textbf{G}$, $|\textbf{G}|$, is then the set of the material designs in it, i.e., $|\textbf{G}|= \{|\de{D} |_{\textbf{G}} \, |\, \de{D} \in \textbf{G} \}$.  
 \end{defi}
An important construction w.r.t. incarnation and generation of behaviours is the $\da$-shortening of a set of designs:

\begin{defi}\label{Dda} 
A $\da$-\textbf{shorten} of a chronicle $\de{c}$ is either $\de{c}$ or a prefix of $\de{c}$ ended by $\da$, i.e., $\de{c}_{1} \da$, when $\de{c}=\de{c}_{1} \de{c}_{2}$ and $\de{c}_1$ ends with a negative action. Given a set of designs $E$ we define its $\da$-\textbf{shortening} $E^{\da}$ as the set of designs obtained from $E$ by $\da$-\textbf{shortening} chronicles.
\end{defi}

\begin{example}
Let $	E=\{\de{D}\}$ where 
\scalebox{.8}
{$\de{D}= \shortstack{ \shortstack{  $\hrulefill_{\emptyset}$  \\ $\vdash \alpha.1.0$ \\\hrulefill\\ $\alpha.1 \vdash$    } \hspace{1em} \shortstack{  $\hrulefill_{\emptyset}$  \\ $\vdash \alpha.3.1$ \\\hrulefill\\ $\alpha.3 \vdash$} \\\hrulefill\\ $  \vdash \alpha$}$}. Then $E^{\da}$ contains $\de{D}$ and the following designs: \scalebox{.8}
{$\shortstack{ \shortstack{  $\hrulefill_{\da}$  \\ $\vdash \alpha.1.0$ \\\hrulefill\\ $\alpha.1 \vdash$    } \hspace{1em} \shortstack{  $\hrulefill_{\da}$  \\ $\vdash \alpha.3.1$ \\\hrulefill\\ $\alpha.3 \vdash$} \\\hrulefill\\ $  \vdash  \alpha$ }$, \hspace{1em} $\shortstack{ \shortstack{  $\hrulefill_{\emptyset}$  \\ $\vdash \alpha.1.0$ \\\hrulefill\\ $\alpha.1 \vdash$    } \hspace{1em} \shortstack{  $\hrulefill_{\da}$  \\ $\vdash \alpha.3.1$ \\\hrulefill\\ $\alpha.3 \vdash$} \\\hrulefill\\ $ \vdash  \alpha$ }$, \hspace{1em} $\shortstack{ \shortstack{  $\hrulefill_{\da}$  \\ $\vdash \alpha.1.0$ \\\hrulefill\\ $\alpha.1 \vdash$    } \hspace{1em} \shortstack{  $\hrulefill_{\emptyset}$  \\ $\vdash \alpha.3.1$ \\\hrulefill\\ $\alpha.3 \vdash$} \\\hrulefill\\ $  \vdash \alpha$ }$, \hspace{1em} \shortstack{    $\hrulefill_{\da}$\\$ \vdash \alpha$}.}
 
\end{example}

\begin{lemme}\label{Edai} 
Given a set $E$ of designs on the same base, $E^{\da} \subseteq E^{\perp\perp}$.
\end{lemme}
\begin{proof}
If $\de{D} \in E^{\da}$ then there exists $\de{E} \in E$ s.t. $\de{D}$ is obtained from $\de{E}$ by $\da$-shortening chronicles. By definition $\de{E} \perp \de{F}$ for all $\de{F} \in E^{\perp}$. $\de{D}$ is a $\da$-shortening of $\de{E}$, therefore $\de{D} \perp \de{F}$ for all $\de{F} \in E^{\perp}$, i.e., $\de{D} \in E^{\perp\perp}$. Thus $E^{\da} \subseteq E^{\perp\perp}$.  
\end{proof}

\begin{lemme}\label{matdes}
Let $\de{D}$ be a material design in a behaviour $\textbf{G}$, then all the designs in its $\da$-shortening are material in $\textbf{G}$, i.e., if $\de{D} \in |\textbf{G}|$ then $\{ \de{D}  \}^{\da}$ $\subseteq | \textbf{G} |$.

\end{lemme}

\begin{proof}
Let $\de{E} \in \{ \de{D}  \}^{\da}$, then either $\de{E}=\de{D}$ (in this case there is nothing to prove) or $\de{E}$ is obtained from $\de{D}$ by $\da$-shortening chronicles. We prove by contradiction that $\de{E} \in |\textbf{G}|$. Let $\de{F} \subsetneq \de{E}$ s.t. $\de{F} \in \textbf{G}$, that is $\de{E} \notin |\textbf{G}|$. This means that there exists a negative action $\kappa^{-}$ and a chronicle $\de{c} \in \de{F}$, s.t. $\de{c} \kappa^{-} \in \de{E}$ and $\de{c}\kappa^{-} \notin \de{F}$. $\de{c} \kappa^{-} \in \de{E}$ and $\de{E} \in \{\de{D} \}^{\da}$, then $\de{c} \kappa^{-} \in \de{D}$. For all $\de{G} \in \textbf{G}^{\perp}$, $\de{G} \perp \de{E}$ and $\de{G} \perp \de{F}$. This means that the computation of $\llbracket \de{E}, \de{G} \rrbracket$ does not use $\kappa^-$. Let $\de{F}'$ be $\de{D}$ without the chronicle $\de{c} \kappa^{-}$ and its extensions, then $\de{F}' \subsetneq \de{D}$ and $\de{G} \perp \de{F}' $ for all $\de{G} \in \textbf{G}^{\perp}$, i.e., $\de{F}' \in \textbf{G}$. Then $\de{D}$ is not material in $\textbf{G}$ (contradiction). 
\end{proof}

\begin{lemme}\label{matset}
Let $E$ be a subset of the incarnation of a behaviour $\textbf{G}$, then all the designs in its $\da$-shortening are material in $\textbf{G}$, i.e., if $E \subseteq |\textbf{G}|$, then $E^{\da}$ $\subseteq |\textbf{G}|$.
\end{lemme}

\begin{proof}
$E^{\da} = \bigcup_{\de{E} \in E} \{\de{E}\}^{\da}$. From Lemma \ref{matdes} $\{\de{E}\}^{\da} \subseteq | \textbf{G}|$ for all $ \de{E}  \in E$. Thus $E^{\da} \subseteq |\textbf{G}|$. 
\end{proof}
Now we introduce the notion of \textbf{principal} set of designs. Roughly speaking a set $E$ is principal when it contains enough $\da$-free designs to recover the behaviour $E^{\perp\perp}$, i.e., the $\da$-free generators of $E^{\perp\perp}$. This notion will be central in our representation of Martin-L\"of Type Theory. 

\begin{defi}\label{princ}
A set $E$ of designs is \textbf{principal} when its elements  are $\da$-free and its $\da$-shortening is the incarnation of its biorthogonal, i.e., $|E^{\perp\perp}| =E^{\da}$. 
\end{defi}
In Ludics a behaviour is completely determined by its material designs. Moreover $\da$-free designs characterize the representation of MALL proofs in \cite{LocSol}. This notion of principal set looks like a good candidate to represent the notion of canonical terms, indeed in Type Theory a type is completely determined by its canonical terms. 
\subsection{From chronicles to paths}
The incarnation of a set of designs is characterized in \cite{IncLud} introducing the following notions and Proposition. We use this result to prove that some sets of designs are principal.

\begin{defi} 
A \textbf{base of net} $\beta$ is a non-empty finite set of sequents of pairwise disjoint addresses: $\Gamma_1\vdash \Delta_1,...,\Gamma_n\vdash \Delta_n$ such that each $\Gamma_i$ contains exactly one address $\xi_i$, except at most one that may be empty, and the $\Delta_j$ are finite sets. A sequence of actions $\de{s}$ is based on $\beta$ if an action of $\de{s}$ either is hereditarily justified\footnotemark\footnotetext{An action $\kappa$ is justified by the action $\kappa'$ when the address of $\kappa$ is built from the address of $\kappa '$. $\kappa$ and $\kappa '$ always have opposite polarity. For instance $(+, \xi.0.2, \{ 0\})$ is justified by $(-, \xi.0, \{2 \})$ and $(-, \alpha. 3, \{ 2\})$ is justified by $(+, \alpha, \{ 3\})$.} by an element of one of the sets $\Gamma_i$ or $\Delta_i$, or is the daimon and in this case is the last action of $\de{s}$. An action is initial if its address is an element of one of the sets $\Gamma_i$ or $\Delta_i$. 
\\Let $\de{s}$ be a sequece of actions based on $\beta$, the \textbf{view}    $\view{\de{s}}$ is the subsequence of $\de{s}$ defined as follows: $\view{\epsilon} = \epsilon$; $\view{\kappa} = \kappa$; $\view{w \kappa^+} = \view{w} \kappa^+$; $\view{w \kappa^-}= \view{w_0} \kappa^-$ where $w_0$ either is empty if $\kappa^-$ is initial or is the prefix of $w$ ending with the positive action which justifies $\kappa^-$.

\end{defi}

\begin{defi} \label{path} 

A $\textbf{path}$ $\de{p}$ based on $\beta$ is a finite sequence of actions based on $\beta$ such that
\\ \emph{Alternation}: The polarity of actions alternates between positive and negative.
\\ \emph{Justification}: A proper action is either justified, i.e., its address is built by one of the previous actions in the sequence, or it is called initial with a address in one of the $\Gamma_i$ (resp. $\Delta_i$) if the action is negative (resp. positive).
\\ \emph{Negative jump} (no jump on positive actions) : Let $\de{q} \kappa$ be a prefix of $\de{p}$. If $\kappa$ is a positive proper action justified by a negative action $\kappa'$ then $\kappa' \in $\view{ $\de{q}$ }. If $\kappa$ is an initial positive proper action then its address belongs to one $\Delta_i$ and  either $\kappa$ is the first action of $\de{p}$ and $\Gamma_i$ is empty, or $\kappa$ is immediately preceded in $\de{p}$ by a negative action with a address hereditarily justified by an element of $\Gamma_i \cup \Delta_i$.
\\ \emph{Linearity}: Actions have distinct addresses. 
\\ \emph{Daimon}: If present, a daimon ends the path. If it is the first action in the $\de{p}$ then one of the $\Gamma_i$ is empty. 
\\ \emph{Totality}: If there exists an empty $\Gamma_i$, then $\de{p}$ is non empty and begins either with $\da$ or with a positive action with a address in $\Delta_i$.
\end{defi}

\begin{rem}
Let $\kappa$ be a positive proper action justified by a negative action $\kappa'$. $\kappa' \in \view{ \de{p}}$ iff there is a sequence $\alpha_n^- \alpha_n^+...\alpha_0^- \alpha_0^+ $ with $\alpha_0^+= \kappa$,  $\alpha_0^-= \kappa'$ such that $\alpha_i^-$ immediately precedes $\alpha_i^+$ in $\de{p}$ and $\alpha_{i+1}^+$ justifies $\alpha_i^-$.
\\\\We remark that a \textbf{chronicle} $\de{c}$ is a path such that each negative action is justified by the immediately precedent action. 
\end{rem}

\begin{defi} 
Two paths $\de{p}_1,\de{p}_2$ on the same base are coherent, noted $\de{p}_{1} \coh \de{p}_{2}$, when:

\begin{varitemize}  
  \itemsep0em

\item their first action have same polarity: either positive and the first actions are the same or negative;

\item for all sequences $w_1\kappa_1^+$ and $w_2 \kappa_2^+$ respectively prefixes of $\de{p}_1$ and $\de{p}_2$: if $\view{ w_1} = \view{ w_2}$ then $\kappa_1^+ = \kappa_2^+$;

\item for all sequences $w_1\kappa_1^-$ and $w_2 \kappa_2^-$ respectively prefixes of $\de{p}_1$ and $\de{p}_2$, let $w_1^0$ (resp. $w_2^0$) be either the empty sequence if $\kappa_1^-$ (resp. $\kappa_2^-$) is initial or the prefix of $\de{p}_1$ (resp. $\de{p}_2$) ending by the justification of $\kappa_1^-$ (resp. $\kappa_2^-$),
  
    \begin{varitemize}     
  \item if $\view{ w_1^0} = \view{ w_2^0}$ and $\kappa_1^- $ and $\kappa_2^-$ have distinct addresses then for all actions $\sigma_1$ and $\sigma_2$ such that $w_1 \kappa_1^-  w_1' \sigma_1$ and $w_2 \kappa_2^-  w_2' \sigma_2$ are respectively prefixes of $\de{p}_1$ and $\de{p}_2$, and such that $\kappa_1^- \in  \view{w_1 \kappa_1^-  w_1' \sigma_1}$ and $\kappa_2^- \in  \view{w_2 \kappa_2^-  w_2' \sigma_2}$, $\sigma_1$ and $\sigma_2$ have distinct addresses. 
    
  \end{varitemize}

\end{varitemize}

\end{defi}
If $\de{p}_1 \coh \de{p}_2$ then in particular either one extends the other or they first differ on negative actions.

\begin{defi}\label{shor}
Given a path $\de{p}$, a $\da$-shorten of $\de{p}$ is either $\de{p}$ or a prefix of $\de{p}$ ended by $\da$, i.e., $\de{p}_{1} \da$, when $\de{p}=\de{p}_{1} \de{p}_{2}$ and $\de{p}_{1}$ ends with a negative action.
\end{defi}

\begin{example}
Let $\de{p}= (+, \xi, \{1\}) (-, \xi.1, \{0\})(+ , \xi.1.0, \{3\})(-, \xi.1.0.3, \{2\})\\(+, \xi.1.0.3.2, \emptyset)$. Then the $\da$-shortens of $\de{p}$ are: $\de{p}$, $\da$, $(+, \xi, \{1\}) (-, \xi.1, \{0\}) \da$ and 
$ (+, \xi, \{1\}) (-, \xi.1, \{0\}) $ $(+, \xi.1.0, \{3\})(-, \xi.1.0.3, \{2\}) \da$.  
\end{example}
We say that $\de{p}$ is a path of a design $\de{D}$ when the views of all the prefixes of $\de{p}$ are chronicles of $\de{D}$.

\begin{defi}\label{tild}
Given a $\da$-free (or \textbf{proper}) path $\de{p}$ of a certain design we define the \textbf{opposite} of $\de{p}$, $\overline{\de{p}}$ as the sequence of actions obtained from $\de{p}$ by changing polarity of each action: $\overline{\epsilon} = \epsilon$,  $\overline{\de{p} (+, \xi, I)} = \overline{\de{p}} (-, \xi, I)$, $\overline{ \de{p} (-, \xi, I) }= \overline{\de{p}} (+, \xi, I) $. We define the \textbf{dual} $ \widetilde{\de{p}}$ of $\de{p}$ as follows: 
\begin{varitemize}

\item $\widetilde{w \da}$ $=\overline{w}$, $\widetilde{w \kappa^{+}}= \overline{w \kappa^{+}} \da$ if $\kappa^{+}$ is positive and $\kappa^{+}\neq \da$, 

\item $\widetilde{w \kappa^{-}}= \overline{w \kappa^{-}}$ for all negative action $\kappa^{-}$.

\end{varitemize}
\end{defi}
Given a path $\de{p}$, $\widetilde{\de{p}}$ is not always a path, as showed in the following example.
\begin{example}
Let $\de{p} = (+, \xi, \{ 0 \}) (-, \xi.0, \{ 1\} ) (+, \sigma, \{ 1\})$, then $\widetilde{\de{p}} = (-, \xi, \{0 \})$ 
\\$(+, \xi.0, \{ 1\} ) (-, \sigma, \{ 1\}) \da$ which is not a path because of the action $(-, \sigma, \{ 1\})$: it is a negative action but it is neither an initial action nor justified. 
\end{example}
Given a set E of designs on the same base, a \textbf{visitable} path in $E$ is a sequence of actions $\de{p}$ in a design $\de{D} \in E$ which are visited during a normalization with a net of designs of $E^{\perp}$.  Visitable paths correspond to the notion of plays in Hyland-Ong-Nickau game semantics. A characterization of \textbf{visitable} paths is given in \cite{IncLud}: to be visitable in $E$ a path must be such that its dual $\widetilde{\de{p}}$ is a path and for all prefix $w \kappa^{-}$ of $\de{p}$, for all $ \de{D} \in E$, if $w$ is a path of $\de{D}$ then $w \kappa^{-}$ is a path of $\de{D}$.\\\\ 
\textbf{Notation}: We denote with $\view{\view{\de{p}}}$ the set of views of the (non empty) prefixes of $\de{p}$. Given a set $E$ of designs on the same base, $P_{E}$ and $V_E$ respectively denote the set of paths and the set of visitable paths of $E$. Given a set $C$ of paths, $\widetilde{C}=\{  \widetilde{\de{p}} : \de{p} \in C \}$ and $\view{ \view{ \widetilde{C} } } =\{ \view{ \view{ \widetilde{\de{p}} } }: \widetilde{\de{p}} \in \widetilde{C}  \}$, i.e., it is the set of views of the prefixes of paths of $\widetilde{C}$.

\begin{defi} \label{finstablesat}
Let $E$ be a set of designs based on $\beta$ and $C$ a set of paths of designs of $E$. $C$ is \textbf{finite-stable} when for all strictly increasing sequence $(\de{p}_{n})$ of elements of $C$, if $\bigcup \view{ \view{ \de{p}_{n} } }$  is included in a design of $E$ then the sequence $(\de{p}_{n})$ is finite. $C$ is \textbf{saturated} when for all prefix $\de{q}$ of an element of $C$ such that $\de{q}\kappa^{+} \in V_{E}$ ($\kappa^+ \neq \da$) we have that $\de{q} \kappa^{+}$ is a prefix of an element of $C$.
\end{defi}

\begin{prop}\label{compinc}(5.17, \cite{IncLud})\\
Let $E$ be a set of designs based on $\beta$. The incarnation of the behaviour generated by $E$ is computed applying the following steps:
\begin{varitemize}

\item Compute $V_{E}$, the set of visitable paths of $E$.

\item Obtain $|E^{\perp}|$ from the set of maximal cliques $\widetilde{C}$ of $\widetilde{V_{E}}$ such that $C$ is finite-stable and saturated.

\item Compute $V':=V_{|E^{\perp} |}$, the set of visitable paths of $|E^{\perp} |$.

\item Obtain $|E^{\perp\perp}|$ from the set of maximal cliques $\widetilde{C'}$ of $\widetilde{V'}$ such that $C'$ is finite-stable and saturated.

\end{varitemize}
\end{prop}

\subsection{Martin-L\"of Types in Ludics: our Methodology.} \label{method}
In Ludics, terms come before types, as to define a behaviour we have to say what are the designs that belong to it. The corresponding of Martin-L\"of's introduction rule is the definition of which designs represent the canonical terms of a certain type, i.e., the definition of a set of designs that should be principal. The corresponding of formation rule is then to verify that this set is principal. Instead of an elimination rule which says how to manipulate these terms, in Ludics the notion of interaction \cite{LocSol} shows how to manipulate designs, i.e., making them interact between them. The equality rule addresses on canonical terms, it corresponds in Ludics to the equality between a cut-net and its normal form: what remains after eliminating a cut. They are equal in the sense that they yield the same canonical term, i.e., they have the same normalization.\\
We can summarize our framework in the following way: a type $A$ is represented by a behaviour $\textbf{A}^{\alpha}$, on a positive atomic base $\vdash \alpha$ arbitrarily chosen, generated by a \textbf{principal} set of designs $\mathbb{A}^{\alpha}$  (i.e. $\textbf{A}^{\alpha}=(\mathbb{A}^{\alpha})^{\perp\perp}$, $|(\mathbb{A}^{\alpha})^{\perp\perp}| = (\mathbb{A}^{\alpha})^{\da}$ and $\A^{\alpha}$ is $\da$-free). The terms of type $A$ are represented by the elements of $\textbf{A}^{\alpha}$, in particular the canonical terms are the material $\da$-free designs of $\textbf{A}^{\alpha}$, i.e., the designs of $\mathbb{A}^{\alpha}$, while the non canonical terms of type $A$ are the cut-nets $\de{R}$ s.t. their normalization represents a canonical term, i.e., $   \llbracket \de{R}   \rrbracket  \in \A^{\alpha}$. The behaviour $\textbf{A}^{\alpha}$ is ``bigger'' than the type $A$, meaning that $\textbf{A}^{\alpha}$ also contains designs that do not represent any term of type $A$. In the following we omit the superscript $\alpha$ that denotes the base and write $\A$ or $\textbf{A}$, apart where it can be source of misunderstanding.

\section{Simple Types in Ludics}\label{SympleTyp}

In this section we illustrate our proposal, focusing on some simple types. In section \ref{Nat1} we treat the representation of natural numbers, we define a set $\N$ of canonical terms and prove that $\N$ is principal. In section \ref{List} we do the same for lists of length $n$ of natural numbers, with the set $\Ln$. In section \ref{flash} we consider the type arrow, together with some examples of functions on $\N$ and $\Ln$. Only main proofs\footnotemark\footnotetext{You can find other proofs in an extended version on Arxiv. } are given in the paper.

\subsection{Natural Numbers}\label{Nat1} 

A natural number $n \in \mathbb{N}$ is represented by a design $\textbf{n}_\sigma$ on a unary positive base $\vdash \sigma$, in the following inductive way:\\
\begin{minipage}{.45\textwidth}
\scalebox{.9}
{
$\textbf{0}_{\sigma}=\shortstack{ $\hrulefill_{\emptyset}$\\ $\vdash \sigma$}$
}
\scalebox{.9}
{$(\textbf{n\pmb{+}1})_{\sigma}=\shortstack{$ \textbf{n}_{\sigma.0.1} $ \\\hrulefill\\ $ \sigma.0 \vdash $ \\\hrulefill\\ $	\vdash \sigma$ }$.
}
\end{minipage}
\begin{minipage}{.5\textwidth}
 In terms of chronicles: \\
 $\textbf{0}_{\sigma}= \{ (+, \sigma, \emptyset)\}$,
 \\ 
$ (\textbf{n\pmb{+}1})_{\sigma}= (+, \sigma, \{0\}) (-, \sigma.0, \{1\}) \textbf{n}_{\sigma.0.1} $.
\end{minipage}
\\
Abusively, we may write $\textbf{n}$ instead of $\textbf{n}_{\sigma}$. Furthermore we abbreviate the design $\textbf{n}$ as \scalebox{.9}{$\shortstack{ $\hrulefill_{\emptyset}$ \\ $\vdash \sigma.\overline{n}$ \\ \hrulefill\\  \hrulefill\\$ \vdash \sigma$ }$}, where $\overline{0} := \epsilon$ (the empty sequence) and $\overline{n+1}:=\overline{n}.0.1$. We denote with $\N$ the set of designs which represent natural numbers, i.e., $\N=\{ \textbf{n}\, | \, n \in \mathbb{N}\}$. This representation of natural numbers is very close to Terui's representation in Computational Ludics \cite{TeruiComp}, they only differ on the polarity.
\\\\To prove that $\N$ is principal, we prove first some preliminary results.

\begin{lemme}\label{oppch} 
For all $\textbf{n} \in \N$ if $\de{c}$ is a chronicle of $ \textbf{n}$ then $\widetilde{\de{c}}$ is a chronicle.
\end{lemme}

\begin{proof}

For each action $\kappa$ of a certain chronicle of $\textbf{n}$ (that is necessarily proper), the address of $\kappa$ is determined from the immediatly precedent action (in particular all the negative actions in $\textbf{n}$ give rise to only one possible address for the positive action which follows), then when we change the polarity of all the actions of $\de{c}$, we find a sequence of proper actions where the address of each action is determined by the action just before, i.e., a chronicle.   
\end{proof}
The chronicles of two designs of $\N$ are either the same (when they represent the same natural number) or they differ on a positive action on the same base $(+, \sigma. \overline{i}, \{0\})$ and $(+, \sigma.\overline{i}, \emptyset)$ (when one is $\textbf{i}$ and the other is some $\textbf{j}$, where $j > i$).

\begin{prop}\label{chNat} 
Let $\textbf{n},\textbf{n}' \in \N$, and $\de{c} $ be a chronicle of $ \textbf{n}$.

\begin{varitemize}
  
\item If $n=n'$, then $\de{c} \in \textbf{n}'$.  

 \item If $n>n'$, then  
 $\exists \de{c}'$ s.t. $ \de{c}' (+, \sigma.\overline{n'}, \{ 0\}) \preccurlyeq \de{c}$ and\footnotemark\footnotetext{ $\de{c}' (+, \sigma.\overline{n'}, \{ 0\}) \preccurlyeq \de{c}$ denotes that $\de{c}' (+, \sigma.\overline{n'}, \{ 0\})$ is an initial subsequence of $\de{c}$.} $\de{c}' (+, \sigma.\overline{n'}, \emptyset) \in \textbf{n}'$.

 \item If $n<n'$,then $\exists \de{c}'$ s.t. $\de{c} = \de{c}' (+, \sigma.\overline{n}, \emptyset)$ and $\de{c}' (+, \sigma.\overline{n}, \{ 0\}) \in \textbf{n}'$.

 \end{varitemize}

\end{prop}

\begin{proof}
 If $n=n'$ then $\textbf{n}=\textbf{n}'$, so $\de{c} \in \textbf{n}'$. If $n >n'$, let $\de{c}'$ be the prefix of $\de{c}$ which ends with the action $(-, \sigma.\overline{n'-1} 0, \{ 1\})$, s.t. $\de{c}' (+, \sigma. \overline{n}, \{0\}) \preccurlyeq \de{c}$. Then $\de{c}' (+, \sigma.\overline{n'} , \emptyset) \in \textbf{n}'$. If $n<n'$ , if $\de{c}$ is not maximal in $\textbf{n}$, then $\de{c} \in \textbf{n}'$, otherwise its last action is $(+, \sigma.\overline{n} , \emptyset)$. Let $\de{c}'$ be $\de{c}$ without its last action, then $\de{c}=\de{c}' (+, \sigma.\overline{n} , \emptyset)$ and $\de{c}' (+, \sigma.\overline{n}, \{ 0\}) \in \textbf{n}'$.  
\end{proof}

The designs of $\N$ cannot start differ on a negative action, as showed in the following Lemma. 
\begin{lemme}\label{negact} 
For all $\textbf{n},\textbf{n}' \in \N$, if $\kappa_{1}^{-}, \kappa_{2}^{-}$ are negative actions, $\de{c} \kappa_1^- \in \textbf{n}$ and $\de{c} \kappa_2^- \in \textbf{n}'$, then $\kappa_1^- = \kappa_2^-$.  
\end{lemme}

\begin{proof}
Let $\textbf{n} \in \N$. The negative actions in $\textbf{n}$ are $(-, \sigma.\overline{i} 0, \{ 1\})$ for $i=0,...,n-1$, so there does not exist two distinct negative actions with the same address. Remark also that two chronicles of $\textbf{n}$ are one an extension of the other, so they cannot differ on a negative action. 
\end{proof}
Which designs are the elements of $|\N^{\perp}| $? 

\begin{lemme}\label{Fsigma}

$| \N^{\perp} | =  \{\de{F}_{\sigma.\overline{0}} \}^{\da}  $, where $\de{F}_{\sigma. \overline{0}}$ is defined in the following way
\\\\ $\forall i \in \mathbb{N}$, $\de{F}_{\sigma.\overline{i} }:= \shortstack{    \shortstack{ $\hrulefill_{\da}$ \\     $\vdash \phantom{a}$ }   \shortstack{   $\de{F}_{\sigma.\overline{i+1}}$  \\\hrulefill\\  $\vdash \sigma.\overline{i}. 0$ }  \\\hrulefill\\  $ \sigma. \overline{i} \vdash $}$, and in terms of chronicles
\\\\ $\de{F}_{\sigma.\overline{i} }= \{  (-, \sigma. \overline{i}, \emptyset) \da$, $ (-, \sigma. \overline{i}, \{0\})(+, \sigma. \overline{i}. 0, \{1\} ) (-, \sigma. \overline{i+1}, \emptyset) \da$ $\, | \, \forall i \in \mathbb{N} \}$. 

\end{lemme}

\begin{proof}

\begin{varitemize}

\item We prove by induction on $n\in \mathbb{N}$ that  $\de{F}_{\sigma.\overline{0}}$ is orthogonal to all the elements $\textbf{n} \in \N$:

\begin{varitemize}  

\item  $\de{F}_{\sigma.\overline{0}} \perp   \textbf{0}$, because $\textbf{0}= \{ (+, \sigma, \emptyset) \}$ and $\de{F}_{\sigma.\overline{0}}$ contains the chronicle $(-, \sigma, \emptyset) \da$,

\item   if $\de{F}_{\sigma.\overline{0}} \perp \textbf{n}$, then $\de{F}_{\sigma.\overline{0}}$ contains the chronicle $\de{c}=w(-, \sigma.\overline{n}, \emptyset) \da$$ =(-, \sigma, \{ 0\}) $
\\$(+, \sigma.0 , \{1\})...(-, \sigma.\overline{n}, \emptyset) \da$. By definition of $\de{F}_{\sigma.\overline{0}}$ it also contains the chronicle $w (-, \sigma.\overline{n}, \{ 0\})  (+, \sigma.\overline{n}.0, \{1\}) (-, \sigma.\overline{n+1}, \emptyset) \da$ i.e. it is orthogonal to $\textbf{n\pmb{+}1}$. 

\end{varitemize} 
So $\de{F}_{\sigma.\overline{0}} \in \N^{\perp}$ and by definition of $\da$-shortening $\{ \de{F}_{\sigma.\overline{0}}\}^{\da}$ $\subseteq \N^{\perp}$. 

\item We prove by contradiction that $  \de{F}_{\sigma.\overline{0}}  \in |\N^{\perp}|$. If there exists $\de{F} \subsetneq  \de{F}_{\sigma. \overline{0}} $, s.t. $\de{F} \in \N^{\perp}$, that is $\de{F}_{\sigma.\overline{0}} \notin |\N|^{\perp}$, then there exists a chronicle $\de{c}$ s.t. $\de{c} \in  \de{F}_{\sigma. \overline{0}}$ and $\de{c} \notin \de{F}$. By definition of $ \de{F}_{\sigma. \overline{0}}$, either $\de{c}=(-, \sigma , \emptyset)\da$ or $\de{c}= (-, \sigma, \{0\})...(-, \sigma. \overline{n}, \emptyset)\da$ for some $n \in \mathbb{N}$ and, $\de{F}$ does not contain any $\da$-shorten of $\de{c}$. Then\footnotemark\footnotetext{If $\de{F}$ does not contain $(-, \sigma, \emptyset) \da$, then $\llbracket \de{F}, \textbf{0} \rrbracket \neq \{\da $$\}$. If $\de{F}$ does not contain $(-, \sigma, \{0\})...(-, \sigma. \overline{n}, \emptyset)\da$ for some $n \in \mathbb{N}$, then $\llbracket \de{F}, \textbf{n} \rrbracket \neq \{\da $$\}$} either $\de{F} \notperp \textbf{0}$ or $\de{F} \notperp \textbf{n}$, i.e., $\de{F} \notin \N^{\perp}$ (contradiction).\\
Therefore $   \de{F}_{\sigma.\overline{0}}  \in |\N^{\perp}|$. 

\item Furthermore an incarnation is closed by $\da$-shortening (by Lemma \ref{matdes}), thus $\{\de{F}_{\sigma. \overline{0}}\}^{\da} \subseteq |\N^{\perp}|$. 

\item Now we prove the second inclusion $| \N^{\perp} | \subseteq\{ \de{F}_{\sigma.\overline{0}}\}^{\da} $. 
\\ Let $\de{E} \in |\N^{\perp}$|, this means that $\de{E} \perp \textbf{n}$ for all $\textbf{n} \in \N$ and it is minimal w.r.t. inclusion. $\de{E} \perp \textbf{0}$, then the chronicle $(-, \sigma, \emptyset)\da$ belongs to $\de{E}$. $\de{E} \perp \textbf{n}$ for all $n>0$, then $\de{E}$ must contain the chronicles  $C=\{(-, \sigma, \{ 0\}) (+ \sigma.0, \{ 1\})... (-, \sigma.\overline{n}, \emptyset)\da$ |  $n>0\}^{\da}$. $\de{E}$ does not contain other chronicles apart $C \cup \{  (-, \sigma, \emptyset) \da$ $\}$, otherwise there would exist a design $\de{E}' \subsetneq \de{E}$, $\de{E}' \in \{\de{F}_{\sigma. \overline{0}}\}^{\da}$ s.t. $\de{E}' \in \N^{\perp}$, i.e., $\de{E}$ is not material in $\N^{\perp}$ (contradiction). Then the only chronicles of $\de{E}$ are $C \cup  \{  (-, \sigma, \emptyset), \da$ $\}$, i.e., $\de{E} \in  \{\de{F}_{\sigma. \overline{0}}\}^{\da}$.
\\Thus $|\N^{\perp}|  \subseteq \{\de{F}_{\sigma.\overline{0}}\}^{\da}  $.  
\\Therefore $\{\de{F}_{\sigma.\overline{0}} \}^{\da}$ $=| \N^{\perp}|$.

\end{varitemize}
\end{proof}
Besides the elements of $\N$ which correspond to the canonical terms of type $\mathbb{N}$, the behavior $\N^{\perp\perp}$ contains some elements which correspond to non canonical ones, but it contains also some designs which do not represent any term of type $\mathbb{N}$.  

\begin{example} \label{noncan}
The design $\de{D}= \shortstack{  \shortstack{$\hrulefill_{\emptyset}$\\ $\vdash \sigma.0.1$   \\\hrulefill\\ $\sigma.0 \vdash$}  \hspace{1em}  \shortstack{$\hrulefill_{\emptyset}$ \\$\vdash \sigma.1.1$   \\\hrulefill\\ $\sigma.1\vdash$}   \\\hrulefill\\ $\vdash \sigma$  }$ does not belong to $\N$, i.e., it does not represent a canonical term of type $\mathbb{N}$. $\de{D} \in \N^{\perp\perp}$, but it is not a net $\de{R}$ s.t. its normalization belongs to $\N$, i.e., it does not represent a non canonical term of type $\mathbb{N}$.
\end{example}

\begin{prop} \label{Nat} 
 $\N$ is principal, i.e., it is $\da$-free and $| \N^{\perp\perp} |= \N^{\da}$.
\end{prop}

\begin{proof}
\begin{varitemize}

\item The fact that $\N$ is $\da$-free follows from the definition of elements of $\N$.

\item We prove by contradiction that $\N \subseteq  | \N^{\perp\perp} |$. \\
Let $\textbf{n} \in \N$, suppose that $\textbf{n} \notin |\N^{\perp\perp}|$, i.e., there exists $\de{E} \subsetneq \textbf{n}$ s.t. $\de{E} \in \N^{\perp\perp}$. Since $\textbf{n} $ contains only one maximal chronicle $\de{c}$, $\de{E}$ contains only one maximal chronicle $\de{c}'$ that is an initial prefix of $\de{c}$. 

\begin{varitemize}

\item If $\textbf{n} = \textbf{0}$, then $\textbf{n}= \{ (+, \sigma, \emptyset) \}$, then $\de{E} = \textbf{n}$. But $\de{E} \subsetneq \textbf{n}$ (contradiction). 

\item Otherwise ($\textbf{n} \neq \textbf{0}$) $\de{c}= \{ (+, \sigma, \{ 0\} ) (-, \sigma. 0 , \{ 1\})...(+, \sigma. \overline{n}, \emptyset) \} $ and since maximal chronicles end with a positive action there exists some $n' < n$ s.t. $\de{c}' = \{ (+, \sigma, \{ 0\} ) (-, \sigma. 0 , \{ 1\})...(+, \sigma. \overline{n}', \{ 0\} ) \} $. We consider the design $\de{F}_{\sigma.\overline{0}}$ defined in Lemma \ref{Fsigma}.\\
 $\de{F}_{\sigma.\overline{0}} \in \N^{\perp}$ and $\de{F}_{\sigma.\overline{0}}  \notperp \de{E}$, indeed $\de{E} = \{ \de{c}'\} $ and $\de{F}_{\sigma.\overline{0}}$ does not contains the chronicle $\widetilde{\de{c}'}$. This means that $\de{E} \notin  \N^{\perp\perp}$. But we supposed $\de{E} \in  \N^{\perp\perp}$ (contradiction).\\
 Then $\N \subseteq |\N^{\perp\perp}|$. 
\end{varitemize} 
Therefore from Lemma \ref{matset} we obtain $ \N^{\da}   \subseteq |\N^{\perp\perp}|$.

\item Now we prove the second inclusion $ |\N^{\perp\perp}| \subseteq  \N^{\da} $. If $\de{D} \in | \N^{\perp\perp}| $ in particular $\de{D} \in \N^{\perp\perp}$, then it is orthogonal to all the elements of $\N^{\perp}$. From Lemma \ref{Fsigma} we have that $ \de{F}_{\sigma.\overline{0}}^{\da}$ $ = |\N^{\perp} |$, then $\de{D} \perp \de{F}_{\sigma. \overline{0}}$. This means that $\de{D}$ contains a chronicle $\de{c}'$ that is a prefix of the chronicle $(+, \sigma, \{0\})  (-, \sigma0 , \{1\})...$
\\$(+, \sigma. \overline{n}, \emptyset) $ (for some $n \in \mathbb{N}$) maybe ended by $\da$. Note that $\{ \de{c}'\} \in \N^{\perp\perp}$. Moreover $\de{D}$ is material in $\N^{\perp\perp}$ i.e. if there exists $\de{E} \subsetneq \de{D}$ s.t. $\de{E} \in \N^{\perp\perp}$, then $\de{E}= \de{D}$. Then $\de{c}'$ is the only chronicle of $\de{D}$ and by definition of $\N$, $\de{D} \in \N^{\da}$. Then $   | \N^{\perp\perp}|    \subseteq   \N^{\da} $. 

\end{varitemize}

Therefore $ \N^{\da} =  |\N^{\perp\perp}| $. 
 \end{proof}

\subsection{Lists}\label{List}
Suppose $A$ a type, $\underline{0}=\epsilon$, $\underline{i+1}:= \underline{i}.1.1$ and $\de{A}^{a_{1}} _{\xi.0.1 }$ the design, based on $\vdash \xi.0.1$, that  represents the element $a_{1}$. We define $\de{D}^{ <a_{1},...,a_{n}>}_{ \xi}$ which represents the list $<a_{1},...,a_{n}>$ (of canonical elements of $A$), on the base $\vdash \xi$, as follows:
\\ \scalebox{.9}{$\de{D}^{\epsilon}_{\xi}=  \shortstack{ $\hrulefill_{\emptyset}$ \\ $\vdash \xi$ }$} is the empty list, \scalebox{.9}{$\de{D}^{<a_{1},...,a_{n}>}_{\xi} = \shortstack{     \shortstack{ $\de{A}^{a_{1}}_{\xi.0.1}$  \\\hrulefill\\ $\xi.0 \vdash$   }     \hspace{1em}  \shortstack{  $\de{D}^{<a_{2},...,a_{n}>}_{\xi.\underline{1}}$ \\\hrulefill\\ $\xi.1 \vdash $ }     \\\hrulefill\\ $  \vdash  \xi$} $} $\forall n>0$.
\\\\We denote with $\Ln$ the set of designs that represent lists of length $n$ of natural numbers, i.e., $\Ln = \{ \de{D}^{ <a_{1},...,a_{n}>}_{ \xi} | a_1,...,a_n \in \mathbb{N} \} $. We just defined the canonical terms of $(\Ln)^{\perp\perp}$.

 We can generalize some results that we have proved for $\N$ to the case of $\Ln$.

\begin{lemme}\label{oppchLn} 
For all $\de{D} \in \Ln$ if $\de{c}$ is a chronicle of $\de{D}$ then $\widetilde{\de{c}}$ is a chronicle.
\end{lemme}

\begin{proof}
Similar as for Lemma \ref{oppch}.
\end{proof}

\begin{prop}\label{chLn} 

Let $\de{D},\de{D'}$ be two elements of $\Ln$ and $\de{c}$ a chronicle of $\de{D}$, then one of the following holds:

\begin{varitemize}  

\item $\de{c} \in \de{D'}$ or 

\item  $\exists i,j \in \mathbb{N}$, $\exists \de{c}'$ s.t. either ($\de{c}' (+, \xi.\underline{i}.0.1.\overline{j}, \{ 0\} ) \preccurlyeq \de{c}$ and $\de{c}' (+, \xi.\underline{i}.0.1.\overline{j}, \emptyset)  \in \de{D}'$) or ($\de{c}= \de{c}' (+, \xi.\underline{i}.0.1\overline{j}, \emptyset)$  and $\de{c}' (+, \xi.\underline{i}. 0.1.\overline{j}, \{ 0\} ) \in \de{D}'$).

\end{varitemize} 

\end{prop}

\begin{proof}
Let $\de{D}$ and $\de{D'}$ represent respectively the lists $<a_{0},...,a_{n-1}>$ and $<a_{0}',...,a_{n-1}'>$ and $\de{c} \in \de{D}$. 
\begin{enumerate}

\item If $\de{c}$ is a prefix of the chronicle which represents the empty list with last action $(+, \xi\underline{n}, \emptyset)$, then this chronicle is common to all the elements of $\Ln$, so $\de{c} \in \de{D'}$.

\item If $\de{c}$ is a prefix of the chronicle which represents an element $a_{i}$ of the list. 
\\If $a_{i}=a_{i}'$ then $\de{c} \in \de{D'}$.
\\If $a_{i} \neq a_{i}'$ there are two cases:

\begin{varitemize}    

\item if $a_{i} > a_{i}'$ let $\de{c}'$ be the subchronicle of $\de{c}$ which ends with the action $(-, \xi.\underline{i}.0.1.\overline{a_{i}' -1}.0 , \{ 1\})$ .\\ 
Then $\de{c}' (+, \xi.\underline{i}.0.1.\overline{a_{i}'}  , \emptyset) \in \de{D'}$ ( it follows from Proposition \ref{chNat} and the fact that $a_{i}'$ is represented in the position $i$ in $\de{D'}$). 

\item  if $a_{i}' > a_{i}$ then

\begin{varitemize}  

\item if $\de{c}$ is not maximal in $\de{D}$, i.e., it does not end with the action $(+, \xi.\underline{i}.0.1.\overline{a_{i}}, \emptyset )$, then $\de{c} \in \de{D}'$ (from Proposition \ref{chNat}),

\item otherwise $\de{c}$ ends with the action $(+, \xi.\underline{i}.0.1.\overline{a_{i}}, \emptyset )$. Let $\de{c}'$ be $\de{c}$ without its last action, i.e., $\de{c}=\de{c}'  (+, \xi.\underline{i} .0.1.\overline{a_{i}}, \emptyset)$ and $\de{c}' (+, \xi.\underline{i}.0.1.\overline{a_{i}}, \{ 0\}) \in \de{D'}$.   

\end{varitemize}

\end{varitemize}

\end{enumerate}
\end{proof}
%
%
The chronicles of the elements of $\Ln$ have a particular form, what about the paths of $\Ln$? 

\begin{lemme} \label{chF}  
If $\de{c}$ is a chronicle of $\de{E} \in \Ln$ and a subsequence of a path of  another design $\de{F} \in \Ln$, then $\de{c}$ is a chronicle of $\de{F}$.

\end{lemme}

\begin{proof}
It follows from Proposition \ref{chLn}.
\end{proof}

\begin{lemme} \label{kpos}
Let  $\de{E,F} \in \Ln$, let $\de{q}$ be a path of $\de{E}$ and $\de{F}$, let $\kappa$ an action such that $\de{q} \kappa$ is a path of $\de{E}$ but not of $\de{F}$. Then $\kappa$ is positive.

\end{lemme}

\begin{proof}
 Let $q$ be a path of $\de{E}$ and $\de{F}$, $q \kappa$ a path of $\de{E}$ and not of $\de{F}$. We prove by contradiction that $\kappa$ is positive. Suppose that $\kappa$ is negative, then $\view{ \de{q} \kappa} = \view{ \de{q}_1} \kappa_0^+ \kappa$, where $\kappa_{0}^{+}$ justifies $\kappa$, $\view{ \de{q}_1 \kappa_0^+ }$  is a chronicle $\de{c}_1$ of $\de{E}$, $\de{c}_1 $ is a subsequence of $\de{q}$ and $\de{q}$ is a path of $\de{F}$. Then from Lemma \ref{chF}, $\de{c}_1$ is a chronicle of $\de{F}$. The justifier of $\kappa$ is the last action of $\de{c}_1$, in $\Ln$ there are never two distinct negative actions on the same address, then $\de{c}_1 \kappa$ is a chronicle of $\de{F}$. Moreover $\view{ \de{q} \kappa} =  \de{c}_1 \kappa \in \de{F}$ and $\de{q}$ is a path of $\de{F}$, then $\de{q} \kappa$ is a path of $\de{F}$ (contradiction). Therefore $\kappa$ must be positive.
\end{proof}

\begin{prop} \label{pathLn} 
Let $\de{E ,F}$ be two elements of $\Ln$ and $\de{p}$ a path of $\de{E}$. Then one of the following holds:

\begin{varitemize}  
\item  $\de{p}$ is a path of $\de{F}$

\item there exists two positive actions $\kappa, \kappa'$ on the same address, $\exists i \in \mathbb{N}$ s.t. for some prefix $\de{q}$ of $\de{p}$, $\de{q} \kappa$ is a prefix of $\de{p}$, $\de{q} \kappa' $ is a path of $\de{F}$, and either ($\kappa= (+, \xi.\underline{i}.0.1.\overline{a_{i}}, \emptyset)$ and $\kappa '= (+, \xi.\underline{i}.0.1.\overline{a_{i}}, \{ 0\})$) or ($\kappa = (+, \xi.\underline{i}.0.1. \overline{a_{i}}, \{ 0\})$ and $\kappa' = (+, \xi.\underline{i}.0.1.\overline{a_{i}}, \emptyset)$). 

\end{varitemize}

\end{prop}

\begin{proof}
If $\de{p}$ is a path of $\de{F}$ there is nothing to prove, so let $\de{p} \notin \de{F}$. 
\\If $n=0$ the only element of $\Ln$ is the empty list, then $\de{E}=\de{F}$ and $\de{p}$ is a path of $\de{F}$ (as above). 
\\If $n \neq 0$,  $\de{E} \neq \de{F}$  (otherwise is like the first case) and $\de{p}$ is not a path of $\de{F}$, there exists an action $\kappa$ such that  $\de{q} \kappa$ is a prefix of $\de{p} $, $\de{q}$ is a path of $\de{F}$ and $\de{q} \kappa$ is not  a path of $\de{F}$. $\de{q}$ is not empty because $n\neq 0$ and then all the chronicles of the elements of $\Ln$ start with the same positive action $(+, \xi,\{0,1\})$, so all paths of $\de{E}$ and $\de{F}$ have at least their first action in common. Then $\kappa$ is not an initial action. From Lemma \ref{kpos}, $\kappa$ is a positive proper (because the elements of $\Ln$ are $\da$-free) action. $\view{ \de{q} \kappa}$ is a chronicle $\de{c}_1$ of $\de{E}$ which ends with a positive action, then from Proposition \ref{chLn} there exists $ \de{c}_1'$ s.t. either ($\de{c}_1'   (+, \xi. \underline{i}.0.1.\overline{a_{i}}, \{0\})$ is a subsequence of $\de{c}_{1}$ and $\de{c}_1' (+, \xi.\underline{i}.0.1.\overline{a_{i}}, \emptyset) \in \de{F}$) or ($\de{c}_{1}=\de{c}_1' (+, \xi. \underline{i}.0.1.\overline{a_{i}}, \emptyset)$ and $\de{c}_1' (+, \xi.\underline{i}.0.1.\overline{a_{i}}, \{0\})  \in \de{F})$. Indeed $\de{c}_1$ is not a chronicle of $\de{F}$, otherwise $\de{q} \kappa$ would be a path of $\de{F}$. The thesis follows taking $\de{q}=\de{c}_{1}'$.
\end{proof}
The duals of the paths of $\Ln$ are still paths as proved in the following lemma. 

\begin{lemme} \label{oppathLn} 
If $\de{p}$ is a path of $\Ln$ then $\widetilde{\de{p}}$ is chronicle, hence a path. 
\end{lemme}

\begin{proof}
If $\de{p}$ is a path of the empty list, then $\de{p} = (+, \xi, \emptyset)$, $\widetilde{\de{p}}= (-, \xi, \emptyset) \da$ and $\widetilde{\de{p}}$ is still a path.Otherwise $\de{p}$ has a particular form $\de{p}=(+, \xi, \{0,1 \}) \kappa'_1 \kappa_1  \kappa'_2 \kappa_2....$ where $\kappa'_i$ is negative and justifies $\kappa_i$. When we change the polarities to obtain $\widetilde{\de{p}}$ we find that each negative (non initial) action of $\widetilde{\de{p}}$ is justified by the action which immediatly precedes it in $\widetilde{\de{p}}$. This means that $\widetilde{\de{p}}$ is a chronicle. 
\end{proof} 
Are the paths of $\Ln$ all visitable in $\Ln$? The following Lemma answers affirmatively to this question. 

\begin{lemme} \label{vispathLn} 
All the paths of $\Ln$ are visitable in $\Ln$. 
\end{lemme}

\begin{proof}
Let $\de{p}$ be a path of $ \Ln$, $w \kappa^-$ a prefix of $\de{p}$ such that there exists $ \de{D} \in \Ln$ and $w$ is a path of $\de{D}$. There exists $\de{E} \in \Ln$ s.t.  $w \kappa^-$ is a path of $\de{E}$. If $w \kappa^-$ is not a path of $\de{D}$, from Lemma \ref{kpos} $\kappa^-$ would be positive (contradiction). Therefore $w \kappa^-$ is a path of $\de{D}$. This means that $\de{p}$ is visitable.
\end{proof}

\begin{lemme}\label{pathcov}
Given a design $\de{E} \in \Ln$, there exists a path $\de{p}$ which covers all the actions of $\de{E}$.
\end{lemme}

\begin{proof}
Let $C= \{\de{c}_{1},...,\de{c}_{n+1} \}$  be an enumeration of the maximal chronicles of $\de{E}$. Given two distinct elements $\de{c}_{i}\neq \de{c}_{j}$ of $C$ they start differ on a negative action $\kappa_{ij}$ of $\de{c}_j$,
we denote as $\de{c}_j'$ the rest of $\de{c_j}$ after $\kappa_{ij}$, i.e., $\de{c}_j= w \kappa_{ij} \de{c}_{j}' $. We define a sequence of actions $\de{p}$ as $\de{c}_{1} \kappa_{12} \de{c}_{2}' \kappa_{23}\de{c}_{3}'....\de{c}_{m}'$. The idea is to jump from a chronicle to the other starting from the first action they differ on. By definition the sequence $\de{p}$ covers all the actions of $\de{E}$. Now we prove that $\de{p}$ is a path of $\de{E}$: by construction, it is alternated and it holds linearity, daimon and totality (see Definition \ref{path}). Suppose that $\de{q} \kappa^{+}$ is a prefix of $\de{p}$, then either $\kappa^{+}$ is initial or there exists a negative action $\kappa_0^{-}$ which justifies $\kappa^{+}$. In the latter case $\kappa_0^{-}$ is immediately before $\kappa^{+}$, i.e., it is the last action of $\de{q}$. Then by definition of view, $\kappa_0^{-} \in \view{q}$. Thus $\de{p}$ is a path. 
\end{proof}
For each $\de{D} \in \Ln$ there are several paths which cover all the actions of $\de{D}$. These paths only differ on the order of their actions. Are they coherent? Yes, indeed all paths that belong to the same design are pairwise coherent.

 \begin{lemme}\label{cohpath} 
Let $\de{E}, \de{F}$ be two distinct designs of $\Ln$, let $\de{p}$ a path (resp. $\de{q}$) that covers $\de{E}$ (resp. $\de{F}$) (following the same order to visit their chronicles), then $\de{p}$ and $\de{q}$ are not coherent whereas $\widetilde{\de{p}}$ and $\widetilde{\de{q}}$ are coherent. 
\end{lemme}

\begin{proof}
Let $\de{E}$ and $\de{F}$ respectively represent the distinct lists $<a_1,...,a_n>$ and $<a'_1,...,a'_n>$, then there exists $i\inÊ\{ 1,...,n\}$ such that $a_i \neq a'_i$. Having seen the structure of the elements of $\Ln$, there exists a chronicle $\de{c} \in \de{E, F}$ such that $\de{c} (+, \xi.\underline{i}.0.1.\overline{a_{i}}, \emptyset) \in \de{E}$ and $\de{c} (+, \xi. \underline{i}.0.1.\overline{a_{i}},\{ 0\}) \in \de{F}$ (or viceversa). $\de{p}$ and $\de{q}$ cover all the actions of $\de{E}$ and $\de{F}$, this means that there exists a subsequence $w_1 \kappa^{+}_1$ of $\de{p}$ and a prefix $w_{2} \kappa_{2}^{+}$ of $\de{q}$ s.t. $\view{w_{1}} =\view{w_{2}} =\de{c}$ and $\kappa_{1}^{+} \neq \kappa_{2}^{+}$, in particular $\kappa_{1}^{+} = (+, \xi.\underline{i}.0.1.\overline{a_{i}}, \emptyset)$ and $\kappa_{2}^{+} = (+, \xi.\underline{i}.0.1.\overline{a_{i}}, \{0\})$ (or viceversa). Therefore $\de{p}$ and $\de{q}$ are not coherent.  $\de{p}$ and $\de{q}$ start differ on a positive action on the same address, then $\widetilde{\de{p}}$ and $\widetilde{\de{q}}$ start differ on a negative action on the same address, moreover $\widetilde{\de{p}}$ and $\widetilde{\de{q}}$ are chronicles (from the proof of Lemma \ref{oppathLn}). Therefore $\widetilde{\de{p}}$ and $\widetilde{\de{q}}$ are coherent. 
\end{proof}

\begin{rem} \label{difford}
If $\de{p}$ and $\de{q}$ are two paths which cover all the actions of  a design $\de{E}$ visiting its chronicles following two different orders, then  $\de{p}$ and $\de{q}$ start differ on a negative action. Then their duals start differ on a positive action i.e. $\widetilde{\de{p}}$ and $\widetilde{\de{q}}$ are not coherent.

\end{rem}

\begin{lemme} \label{visLnperp} 

 $\widetilde{  V_{(\Ln)^{\perp} }   }  = V_{ \Ln}$.
 
\end{lemme}

\begin{proof}

Given a permutation $\beta $ of $1,...,n+1$ we define the set $P_{\beta }$ of paths as
\bdm
P_{\beta } = \{ \widetilde{\de{p}}\, |\, \exists \de{E} \in \Ln  \text{ s.t. $\de{p}$ covers all the actions of }  \de{E} \text{ following the order\footnotemark  given by }  \beta  \}.
\edm
\footnotetext{The order followed by the elements of $P_{\beta }$ is a permutation of $\{1,...,n+1 \}$ indeed there are at most $n+1$ maximal chronicles of $\de{E}$.}
It follows from Lemma \ref{cohpath} that elements of $P_{\beta }$ are pairwise coherent. Thus, as a set of pairwise coherent paths forms a design, we can define the design $\de{G}_{\beta }= \view{ \view{ P_{\beta } }  }$, i.e., the set of views of prefixes of the elements of $P_{\beta }$.
Let $G= \{ \de{G} \, | \, \beta$ is a permutation of $1,...,n+1\}$.
By Lemma \ref{vispathLn}, we know that the paths of $\Ln$ are visitable. Furthermore, in $\Ln$, the ramification of a negative action is always the singleton $\{1\}$. Hence for all $\beta$, a positive action of $\de{G}_{\beta }$ is followed by at most one negative action: thus it is not possible to ``jump'' from a chronicle to another, i.e., the only paths of $G$ are its chronicles, which are visitable.
The set of chronicles of $G$ is equal to $\widetilde{V_{\Ln}}$ by definition of $G$.
It follows that $V_{G}= \widetilde{V_{\Ln}}$.
We want to show now that $G= |(\Ln)^{\perp}|$, from which follows that $V_{G}=V_{(\Ln)^{\perp}}$:\\
-- Let $\de{G}_{\beta } \in G$. By definition of $G$, $\de{G} \in (\Ln)^{\perp}$. We prove by contradiction that it is material in it. Suppose that there exists $\de{E} \subsetneq \de{G}_{\beta }$, then there exists a path $\de{p}$ which belongs to $\de{G}_{\beta }$ and $\de{p} \notin \de{E}$ s.t. $\de{p}$ covers a design $\de{L}\in \Ln$. $\de{E} \subsetneq \de{G}_{\beta }$, then $\de{E}$ cannot contain a prefix of $\de{p}$ ended by $\da$. This means that $\de{E}\notperp \de{L}$, i.e. $\de{E} \notin (\Ln)^{\perp}$. Thus $G \subseteq |(\Ln)^{\perp}|$. \\
-- If $\de{F} \in |(\Ln)^{\perp}|$, then $\de{F}$ is composed with pairwise coherent paths $\widetilde{\de{q}}$ s.t. $\de{q} \in V_{\Ln}$. This means that $\de{F}$ contains the duals of the paths which cover distinct elements of $\Ln$ following the same order (Lemma~\ref{cohpath}), i.e. $\de{F} \in G$. Thus $|(\Ln)^{\perp}| \subset G$. \\
Therefore $|(\Ln)^{\perp}|=G$.
\end{proof}

\begin{prop}
$\Ln$ is principal, i.e., it is $\da$-free and $|  (\Ln)^{\perp\perp} | =   (\Ln)^{\da}$. 
\end{prop}

\begin{proof}

\begin{varitemize}

\item The fact that $\Ln$ is $\da$-free follows from its definition.
\item We start proving $  (\Ln)^{\da}$ $ \subseteq |   (\Ln)^{\perp\perp} | $.
\\Let $\de{D} \in \Ln$, then from Lemma \ref{vispathLn} all the paths of $\de{D}$ are visitable. Let $\widetilde{C}$ be the set of paths of $\de{D}$, then $\widetilde{C}$ is a maximal clique of visitable paths. Thus $C \subseteq V_{(\Ln)^{\perp}}$ by means of Lemma \ref{visLnperp}. $C$ is finite stable because it contains a finite number of paths, so each sequence of paths is finite. It is saturated because in $\Ln$ there does not exist two negative actions with the same address, then if $\de{q}$ is a prefix of an element of $C$ such that $\de{q} \kappa^{+} \in V_{(\Ln)^{\perp}}$ ($\kappa^+ \neq \da$) then $\de{q} \kappa^{+}$ is a prefix of an element of $C$ (there is only one possible choice for $\kappa^{+}$). Therefore from Proposition \ref{compinc} $\view{ \view{ \widetilde{C} } }=\de{D} \in |  (\Ln)^{\perp\perp} |$.
\\Thus $\Ln \subseteq |(\Ln)^{\perp\perp}|$.
\\Therefore from Lemma \ref{matset}  follows $  (\Ln)^{\da}$ $ \subseteq |   (\Ln)^{\perp\perp} | $.

\item  Now we prove  $|   (\Ln)^{\perp\perp} | \subseteq   (\Ln)^{\da}$.
\\ If $\de{D} \in |  (\Ln)^{\perp\perp}|$, there exists $C \subset V_{\Ln}^{\perp}$ such that $C$ is finite stable and saturated and $\widetilde{C}$ is a maximal clique of $\widetilde{ V_{(\Ln)^{\perp}} }$ and $\view{ \view{ \widetilde{C} } } =\de{D}$ (Proposition \ref{compinc}). We want to show that $\de{D} \in   (\Ln)^{\da}$. From Lemma \ref{visLnperp}, $\widetilde{  V_{(\Ln)^{\perp} }   }  = V_{ \Ln}$. $\widetilde{C}$ is a maximal clique of $\widetilde{ V_{(\Ln)^{\perp}} }$, then it is a maximal clique of $V_{ \Ln}$. Moreover all paths of $ \Ln$ are visitable (Lemma \ref{vispathLn}). Which paths of $\Ln$ can form a maximal clique? $\widetilde{C}$ cannot contain two distinct paths which respectively cover (all the actions of) two distinct elements of $\Ln$, indeed from Lemma \ref{cohpath} these paths are not coherent between them. Then the elements of $\widetilde{C}$ are all the chronicles of an element of $\Ln$ and their $\da-$shortenings. This means that there exists $\de{E} \in (\Ln)^{\da}$ such that the elements of $\widetilde{C}$ are all the paths of $\de{E}$. Then $\de{D}=\view{ \view{ \widetilde{C} } }  =\de{E} \in (\Ln)^{\da}$. 
\\Therefore $ |   (\Ln)^{\perp\perp} |  \subseteq  (\Ln)^{\da} $. 

\end{varitemize}

\end{proof}

So $\Ln$ represents the canonical terms of type $List_n $(lists of natural numbers of length $n$). This result can be extended to the set $List$ that represents all lists of natural numbers, i.e., $List$ is principal.

\subsection{The type $\to$. }\label{flash}
Given two types $A$ and $B$, $A \to B$ is the set of functions 
from $A$ to $B$. How can we interpret it in Ludics?
Given two behaviours $\textbf{A}$ and $\textbf{B}$, Girard defines \cite{LocSol} the sequent of behaviours $\textbf{A} \vdash \textbf{B}$ as $\{ \de{D} \, |\, \forall \de{A} \in \textbf{A}$ $\llbracket \de{D}, \de{A} \rrbracket \in \textbf{B}\}$, i.e., the designs s.t. their interaction with any element of $\textbf{A}$ gives an element of $\textbf{B}$. It corresponds exactly to the set of functions from $\textbf{A}$ to $\textbf{B}$. We extend this notion also to principal sets of designs. If the behaviours $\textbf{A},\textbf{B}$ are respectively generated by the principal sets $\mathbb{A}^{\alpha},\mathbb{B}^{\beta}$ then we define the set of designs of base $\alpha \vdash \beta$ 
\begin{center}
$  \A^{\alpha} \R \B^{\beta} :=\,\{  \de{D} \text{ minimal w.r.t. inclusion }  |\, \forall \de{A} \in \A^{\alpha}$ $\llbracket \de{D}, \de{A} \rrbracket \in \B^{\beta} \}$.\\ 
\end{center}
In the following we shorten $\A^{\alpha} \R \B^{\beta}$ with $\fr$.

\begin{prop}\label{prefreccia} 
If $\A,\B$ are two principal sets, then $\A\R \B$ is principal. 
\end{prop} 
\begin{proof}
We prove by contradiction that $\fr$ is $\da$-free.
\\ Let $\de{D} \in \A\R \B$ s.t. $\de{D}$ is not $\da$-free. Being $\de{D}$ minimal s.t. $ \forall \de{A} \in \A$, $\llbracket \de{D}, \de{A} \rrbracket \in \B$, all the actions of $\de{D}$ are visited during normalizations $\llbracket \de{D}, \de{A} \rrbracket $. $\de{D}$ contains a daimon, thus there exists $\de{A} \in \A$ s.t. $\llbracket \de{D}, \de{A} \rrbracket = \{ \da$$\}$. But $\llbracket \de{D}, \de{A} \rrbracket \in \B$ and $\B$ is principal, hence it is $\da$-free (contradiction). Thus $\fr$ is $\da$-free. 
\\ Now we prove that $|(\A\R \B)^{\perp\perp} | = (\A\R \B)^{\da}$.
\begin{varitemize}

\item We first prove by contradiction that $\fr \subseteq  |(\fr)^{\perp\perp}|$. Let $\de{D} \in \fr$ and $\de{D} \notin  |(\fr)^{\perp\perp}|$, that is $\exists \de{E} \subsetneq \de{D}$ s.t. $\de{E} \in (\fr)^{\perp\perp}$. 

\begin{varitemize}

\item  Let us prove that $\{ \{ \de{F}, \de{G}\} | \de{F} \in \A, \de{G} \in \B^{\perp}\} \subseteq (\A\R \B)^{\perp}$. If $\de{F} \in \A$, 
\\$\de{G} \in \B^{\perp}$, $\de{D} \in \A\R \B$, by definition of $ \A\R \B$, we have $\llbracket \de{D}, \de{F}\rrbracket \in \B$. Then $ \llbracket \llbracket \de{D}, \de{F}\rrbracket, \de{G} \rrbracket=\{ \da$$ \}$, moreover\footnotemark\footnotetext{See Associativity Theorem \cite{LocSol}.} $\llbracket \de{D}, \de{F} , \de{G}\rrbracket= \llbracket \llbracket \de{D}, \de{F}\rrbracket, \de{G} \rrbracket$, then \\$ \llbracket \de{D}, \de{F} , \de{G}\rrbracket= \{ \da $$ \}$. Thus $\{ \de{F}, \de{G} \}  \in (\A\R \B)^{\perp}$.

\item Then $(\A\R \B)^{\perp\perp} \subseteq \{ \{ \de{F}, \de{G}\} | \de{F} \in \A, \de{G} \in \B^{\perp}\}^{\perp}$.

\item $\llbracket \de{D}, \de{A} \rrbracket \in \B$, $\B$ is principal, then $\llbracket \de{D}, \de{A} \rrbracket \in|  \B^{\perp\perp}| $

\item Since $\de{E} \in (\A \R \B)^{\perp\perp}$,  then for all $\de{A} \in \A$, $\de{B} \in \B^{\perp}$ $\llbracket \de{E}, \de{A}, \de{B}\rrbracket = \{ \da $$\}$. Moreover $\llbracket \de{E}, \de{A}, \de{B}\rrbracket =  \llbracket  \llbracket  \de{E}, \de{A} \rrbracket , \de{B}\rrbracket $, then $\llbracket  \llbracket  \de{E}, \de{A} \rrbracket ,\de{B}\rrbracket = \{ \da$$ \}$, i.e., $ \llbracket  \de{E}, \de{A} \rrbracket \in \B^{\perp\perp}$. $\de{E} \subsetneq \de{D}$, then $\llbracket \de{E} , \de{A} \rrbracket \subsetneq \llbracket \de{D}, \de{A}\rrbracket \in \B$. Moreover $\llbracket \de{D}, \de{A} \rrbracket \in |  \B^{\perp\perp} | $, then $\llbracket \de{E}, \de{A} \rrbracket = \llbracket \de{D}, \de{A} \rrbracket $, i.e., $\llbracket \de{E} , \de{A} \rrbracket \in \B$. But $\de{E} \subsetneq \de{D}$ and $\de{D}$ is minimal s.t. for all $\de{A} \in \A$ its interaction with $\de{A}$ belongs to $\B$ (contradiction).
\end{varitemize}
Therefore $\fr \subseteq  |(\fr)^{\perp\perp}|$.

\item Then from Lemma \ref{matset}  $(\fr)^{\da} \subseteq |(\fr)^{\perp\perp}|$.

\item Let us prove that $|(\A\R \B)^{\perp\perp}| \subseteq (\A\R \B)^{\da}$. Let $\de{E} \in |(\A  \R  \B)^{\perp\perp}|$.
\begin{varitemize}
\item Since $\de{E} \in (\A  \R  \B)^{\perp\perp}$, as above for all $\de{A} \in \A$ we have that $\llbracket \de{E}, \de{A} \rrbracket \in \B^{\perp\perp}$. 
\item We prove by contradiction that for all $\de{A} \in \A$, $\llbracket\de{E}, \de{A} \rrbracket \in \B^{\da}$. 
\\Let $\llbracket\de{E}, \de{A} \rrbracket \notin \B^{\da}$. Since $\llbracket\de{E}, \de{A} \rrbracket \in \B^{\perp\perp}$, $\B$ is principal and $|\B^{\perp\perp}| \subseteq \B^{\perp\perp}$, $\llbracket\de{E}, \de{A} \rrbracket \notin \B^{\da}$ means that there exists a design $ \de{D}_{\de{A}} \subsetneq \llbracket\de{E}, \de{A}\rrbracket$ s.t. $\de{D}_{\de{A}} \in \B^{\da}$.
\\ We remark that there exists a design $\de{E}_{\de{A}}$ s.t. $ \de{E}_{\de{A}} \subsetneq \de{E}$ and $\llbracket \de{E}_{\de{A}}, \de{A} \rrbracket  = \de{D}_{\de{A}}$. Furthermore the set $\bigcup_{\de{A} \in \A }  \de{E}_{\de{A}}$ is a design.
\\$ \bigcup_{\de{A} \in \A}   \de{E}_{\de{A}} \subsetneq \de{E}$, and for all $\de{A} \in \A$ $\llbracket   \bigcup_{\de{A} \in \A}  \de{E}_{\de{A}}, \de{A} \rrbracket= \llbracket  \de{E}_{\de{A}} , \de{A}\rrbracket  \in \B^{\da}$, then there exists $\de{F} \subseteq    \bigcup_{\de{A} \in \A}  \de{E}_{\de{A}}$ minimal s.t. $\forall \de{A} \in \A$ $\llbracket \de{F}, \de{A}  \rrbracket \in \B^{\da}$, i.e., $\de{F} \in \A\R \B^{\da}$. Then\footnotemark\footnotetext{$(\fr)^{\da}$$= \A\ \R \B^{\da}$, because $\A$ and $\B$ are principal sets, thus $\da$-free.} $\de{F} \in (\fr)^{\da}$, thus $ \de{F} \in (\A\R \B)^{\perp\perp}$. 
\\Since $\de{F} \subsetneq \de{E}$, $ \de{F} \in (\A\R \B)^{\perp\perp}$ means that $\de{E} \notin |(\A\R \B)^{\perp\perp}|$, but by hypothesis \\$\de{E} \in |(\fr)^{\perp\perp}|$ (contradiction).\\
Thus for all $\de{A} \in \A$, $\llbracket  \de{E}, \de{A} \rrbracket \in \B^{\da}$. 
\item Let us prove by contradiction that $\de{E}$ is minimal s.t. $\forall \de{A} \in \A$, $\llbracket \de{E}, \de{A} \rrbracket \in \B^{\da}$. Let $\de{E}$ not minimal, i.e., there exists $\de{E}' \subsetneq \de{E}$ minimal s.t. $\forall \de{A} \in \A$, $\llbracket \de{E}', \de{A} \rrbracket \in \B^{\da}$. This means that $\de{E}' \in \A \R \B^{\da}$. Then $\de{E}' \in (\fr)^{\perp\perp}$. Since $\de{E}' \subsetneq \de{E}$, this means that $\de{E} \notin |(\fr)^{\perp\perp}|$ (contradiction). 
\end{varitemize}
Then $\de{E} \in \A \R \B^{\da}$, i.e., $\de{E} \in (\fr)^{\da}$.\\
Thus $|(\A\R \B)^{\perp\perp}| \subseteq (\A\R \B)^{\da}$.  
\end{varitemize}
\end{proof}
We want to treat uniformly all behaviours representing a type. We decide then to consider principal sets and behaviours with a positive atomic base, indeed this feature simplifies the representation of higher order types as for instance $(A \to B ) \to C$. The set $\fr$ is principal, but its elements have a negative base. We define then an  encoding which transforms a design on a negative base $\alpha \vdash \beta$ in a design on a positive atomic base. Using this encoding we define a set of designs on a positive, atomic base, that is principal and represents the canonical terms of type $\textbf{A} \to \textbf{B}$. 

\begin{defi}
Given a design $\de{D}$ with base $\alpha \vdash \beta$, where $\alpha = \gamma.0.0.0$ and $\beta= \gamma.0.1$, we define the design 
\\ \scalebox{.9}{
$   \de{D}_{cod} = \,   \shortstack{ $\de{D}$\\ $\hrulefill_{cod}$\\ $\vdash \gamma \phantom{ciao}$ }$ where $\shortstack{ $\hrulefill_{cod}$\\$\vdash \gamma \phantom{ciao}$} = \, \shortstack{  $\gamma.0.0.0 \vdash \gamma.0.1$ \\\hrulefill\\ $\vdash \gamma.0.0, \gamma.0.1$ \\\hrulefill\\ $ \gamma.0 \vdash$ \\\hrulefill\\ $\vdash \gamma$ }  $
}. Since we choose arbitrarily the base of designs we consider negative bases $\alpha \vdash \beta$ where $\alpha$ and $\beta$ have a common prefix. 
\end{defi}
In the following proposition we prove that the set obtained encoding all the designs of a principal set is still principal.

\begin{prop}\label{principalcod}
Given a principal set $\A$, the set $\A_{cod}$ of the encoded elements of $\A$ is principal, i.e., it is $\da$-free and $| (\A_{cod})^{\perp\perp} |= (\A_{cod})^{\da}$. 

\end{prop}

\begin{proof}
\begin{varitemize}

\item $\A_{cod}$ is $\da$-free indeed $\A$ is $\da$-free and the encoding is $\da$-free. 
\\Now we prove $| (\A_{cod})^{\perp\perp} |= (\A_{cod})^{\da}$. 

\item To prove  $ (\A_{cod})^{\da} \subseteq  | (\A_{cod})^{\perp\perp}   | $, we first prove that  $ \A_{cod} \subseteq  | (\A_{cod})^{\perp\perp}   | $.
\\ Let $\de{A}_{cod} \in \A_{cod}$, then $\de{A} \in \A$. By definition of $\da$-shortening $\A \subseteq \A^{\da}$, and $\A$ is principal, i.e., $\A^{\da} = | \A^{\perp\perp} |$, so $\de{A} \in | \A^{\perp\perp} |$ and $\de{A}_{cod} \in  | \A^{\perp\perp} |_{cod}$. Since\footnotemark\footnotetext{ $ |  (\A_{cod} )^{\perp\perp}  | =  | \A^{\perp\perp} |_{cod}$, indeed for all $\de{A} \in \A$, $\de{A}_{cod} =\de{c} \de{D}$, where all the actions of the chronicle $\de{c}$ are justified by the immediately precedent action.}  $ |  (\A_{cod} )^{\perp\perp}  | =  | \A^{\perp\perp} |_{cod}$, we have $\de{A}_{cod} \in  |  (\A_{cod} )^{\perp\perp}  |$ and then $\A_{cod} \subseteq |  (\A_{cod} )^{\perp\perp}  |$.\\
 Therefore from Lemma \ref{matset} $ (\A_{cod})^{\da} \subseteq  | (\A_{cod})^{\perp\perp}   | $.
\item We just need to prove $| (\A_{cod})^{\perp\perp} |\subseteq (\A_{cod})^{\da}$. 
\\ Let $\de{E} \in | (\A_{cod})^{\perp\perp} | $. Since $ |  (\A_{cod} )^{\perp\perp}  | =  | \A^{\perp\perp} |_{cod}$,  $\de{E} \in  | \A^{\perp\perp} |_{cod}$. This means that there exists a design $\de{D} \in  | \A^{\perp\perp} |$ s.t. $\de{E}= \de{D}_{cod}$. $\A$ is principal, so $\de{D} \in (\A)^{\da}$ and then $\de{D}_{cod} \in (\A^{\da})_{cod}$. Moreover $(\A^{\da})_{cod} \subseteq (\A_{cod})^{\da}$, then $\de{D}_{cod} \in (\A_{cod})^{\da}$, i.e., $| (\A_{cod})^{\perp\perp} |\subseteq (\A_{cod})^{\da}$. 

\end{varitemize}
\end{proof}
From now we always suppose that negative designs can be encoded to obtain designs on an atomic positive base.  
\\We can sum up our proposition on the arrow type in the following tabular. 
\begin{center}
\begin{tabular}{| l | l | } \hline
{\textbf{Type Theory}} &  {\textbf{Ludics}} \\ \hline 
\hspace{1em}  $t : A \to B $ & \hspace{1em}$t^{\star}= \de{D} \in \fr$\\
\hspace{1em} $r$ non canonical term of type $A \to B$ & \hspace{1em} $\de{R}$ cut-net s.t. $\llbracket \de{R} \rrbracket \in \fr $\\
\hspace{1em} $u : A$ & \hspace{1em} $u^{\star} = \de{U} $ s.t. $ \llbracket \de{U} \rrbracket \in \A$\\
\hspace{1em} $(t)u : B$ & \hspace{1em} $((t)u )^{\star}= \llbracket  \de{D}, \de{U} \rrbracket \in \B$\\ \hline
\end{tabular}
\end{center}
In \cite{LocSol} Girard introduces the design $\de{Fax}$ that represents the identity function.

\begin{example}
The design $\de{Fax}_{\sigma \vdash \sigma'}\footnotemark\footnotetext{We shorten it with $\de{Fax}$ in the following.}$ is based on $\sigma \vdash \sigma'$
\begin{center}
$\de{Fax}_{\sigma \vdash \sigma'}  = \quad \shortstack{    \shortstack{...\\\\\\...}    \hspace{1em}    \shortstack{  $\vdots$ \\  $\de{Fax}_{\sigma' \star i \vdash \sigma \star i}$  \\\hrulefill\\  $\sigma'  i \vdash \sigma i$ \\\hrulefill\\ $\vdash \sigma', \sigma\star I$}   \hspace{1em} \shortstack{...\\\\\\...}   \\$\hrulefill_{P_f (\mathbb{N})}$\\ $\sigma \vdash \sigma'$ }$
\end{center}
\begin{center}
and as a set of chronicles $\de{Fax}_{\sigma \vdash \sigma'} =\{  (-, \sigma, I) (+, \sigma', I ) \de{Fax}_{\sigma 'i \vdash \sigma i} \, | \, I \in P_f (\mathbb{N}) \}$ \vspace{2mm}
 where $\sigma \star I$ denotes $\sigma.1,..., \sigma.n$ if $I = \{ 1,...,n  \}$.
\end{center}
 It corresponds in term of game semantics to the copycat strategy, i.e., at each step we copy the last action of the opponent. 
 \\In the article, to represent the partial identity function we consider a subset of $\de{Fax}$ that only contains the ramifications necessary to interact with the designs of a given set and call this design $Id$. 
 
 
\end{example}

 \begin{example}
We define the function which adds $n \in \mathbb{N}$ as the following design\\
\scalebox{.8}{
$\de{Su}_n= \quad \shortstack{      \shortstack{$\hrulefill_{\emptyset}$\\ $\vdash \beta.\overline{n}$\\\hrulefill\\\hrulefill\\ $\vdash \beta$  }   \hspace{1em}        \shortstack{  Id \\  $\sigma.\overline{1}  \vdash \beta.\overline{n+1} $  \\\hrulefill\\ $ \vdash  \sigma.0, \beta.\overline{n+1}$  \\\hrulefill\\\hrulefill\\ $\vdash \sigma.0, \beta$ } \\\hrulefill\\  $\sigma \vdash \beta$ }$
}
For all $\textbf{m} \in \N$, the net  $ \{ \textbf{m}, \de{Su}_n \}$ is a non canonical term of $\N$, while its normal form $\llbracket \textbf{m}, \de{Su}_n  \rrbracket =\textbf{m\pmb{+}n}$ is canonical. $\de{Su}_n$ tests if $m=0$ (in this case the result is directly $\textbf{n}$), and if $n>0$ makes n steps (to say that the result is at least $n$) and copies the rest of actions of $\textbf{m}$ with $Id$ ($m-1$ steps) to finally have $m+n$.
\end{example}

\begin{example}
The predecessor function is represented by the following design

\begin{center}
$\de{P}= \quad \shortstack{  \shortstack{$\hrulefill_{\da}$\\$\vdash \alpha$} \hspace{1em} \shortstack{$\de{P}_{0}$} \\\hrulefill\\ $\sigma \vdash \alpha$  }\quad $
and $\de{P}_{\overline{i}}= \shortstack{    \shortstack{  $\hrulefill_{\emptyset}$\\ $\vdash \alpha.  \overline{i}$}     \hspace{1em}   \shortstack{   $\de{P}_{ \overline{i+1}}$    \\\hrulefill\\\hrulefill\\   $\vdash \sigma.\overline{i+1}0, \alpha. \overline{i}$   }  \\\hrulefill\\ $\sigma. \overline{i+1} \vdash \alpha.  \overline{i}$ \\\hrulefill\\ $\vdash \sigma.  \overline{i} 0, \alpha.  \overline{i}$  }$ for all $i \in \mathbb{N}$.

\end{center}

Remark that $\llbracket \de{P}, \textbf{0} \rrbracket = \{ \da$$\}$, then $\de{P}$ belongs to $(\N^{\sigma} \R \N^{\alpha})^{\perp\perp}$ but not to $\N^{\sigma} \R \N^{\alpha}$. We have as desired that $\llbracket \de{P}, \textbf{n} \rrbracket = \textbf{n} \pmb{-}\textbf{1}$ for all $\textbf{n} \geq \textbf{1}$.   

\end{example}

\begin{example}~\\
\begin{minipage}{.43\textwidth}
 Let $b \in \mathbb{N}$ represented by the design $\textbf{b}$ on the base $\vdash \sigma.0.0$, let $l \in \Ln$ be $\langle a_{1},...,a_{n}\rangle$. The function which adds $b$ in head position is represented by $\de{Cons}_b$. The net $\{l, \de{Cons}_{b}\}$ is a non canonical term of $(\mathbb{L}_{n+1})^{\perp\perp}$, while its normal form, denoted by $b.l$, is canonical.  
\end{minipage}
~
\begin{minipage}{.28\textwidth}
\scalebox{.8}
{
$\de{Cons}_{b}= \hspace{-2em}\shortstack{  \shortstack{    $\vdots$ \\ $\de{D}^{<b>}_{1,\sigma}$ \\ $\vdash \sigma$  } \hspace{1em}  \shortstack{    \shortstack{  $\vdots$ \\ $\textbf{b}$ \\ $\vdash \sigma.0.0$ \\\hrulefill\\ $\sigma.0 \vdash$   }    \hspace{1em}   \shortstack{    \shortstack{ $\vdots$ \\ Id \\ $ \sigma.1.0.0 \vdash \xi.0$  }  \hspace{1em}   \shortstack{ $\vdots$ \\ Id \\ $ \sigma.10..1 \vdash \xi.1$  }     \\\hrulefill\\ $\vdash \sigma.1.0, \xi.0, \xi.1$ \\\hrulefill\\ $\sigma.1 \vdash \xi.0, \xi.1$  }      \\\hrulefill\\ $\vdash \xi.0, \xi.1, \sigma$     }      \\\hrulefill\\ $\xi \vdash \sigma$   }$
}
\end{minipage}
\end{example}









\begin{example}
The function which eliminates in a list of integers its element in position 2 is represented by the following design 

\begin{center}
\scalebox{.9}{
$\de{El}= \quad \shortstack{    \shortstack{  $\hrulefill_{\emptyset}$\\ $\vdash \sigma$} \hspace{1em}     \shortstack{        \shortstack{ $\vdots$ \\ Id \\ $ \sigma.1 \vdash \xi.0$} \hspace{1em}    \shortstack{ $\vdots$ \\ Id \\ $  \sigma.2.1.1 \vdash \xi.\underline{1}.1$ \\\hrulefill\\ $\vdash \xi.\underline{1}.0, \xi.\underline{1}.1,  \sigma.2.1$ \\\hrulefill\\ $\xi.\underline{1} \vdash \sigma.2.1$   \\\hrulefill\\ $\vdash \xi.1,  \sigma.2.1$ \\\hrulefill\\ $ \sigma.2 \vdash \xi.1$}  \\\hrulefill\\ $\vdash \xi.0, \xi.1,  \sigma$          } \\\hrulefill\\ $\xi \vdash \sigma$  }$
}
\end{center}
Given $\de{D} \in \Ln$, $\llbracket \de{D}, \de{El}\rrbracket$ gives as result the list represented by $\de{D}$, without the element in position 2 (if it exists).

\end{example}

\section{Focus on Dependent Types}\label{DepTyp}
\subsection{The type $(\Pi x \in A) B(x)$}
We generalize to type $(\Pi x \in A) B(x)$ the modelling given for the type $\to$. Let us give a type $A$ and a family of types $(B(x))_{x \in A}$, suppose that they are respectively represented by means of the principal set $\A^{\alpha}$ and the family of principal sets $(\B(x))^{\beta}_{ x \in \A^{\alpha}}$. We represent the canonical terms of type $(\Pi x \in A) B(x)$ by the designs (based on $\alpha \vdash \beta$) of the set $(\Pi x \in \A^{\alpha}) \B(x)^{\beta}$ defined below.
\bdm (\Pi x \in \A^{\alpha}) \B(x)^{\beta} : =\{   \de{D}  \text{ minimal w.r.t. inclusion } \, |\, \forall \de{A} \in \A^{\alpha}, \llbracket\de{A}, \de{D}\rrbracket \in \B(\de{A})^{\beta} \}
\edm
In the following we shorten $ (\Pi x \in \A^{\alpha}) \B(x)^{\beta}$ with $(\Pi x \in \A) \B(x)$. 

\begin{prop}\label{pinter}
Given a principal set $\A$ and a family of principal sets $(\B(x))_{x \in \A}$, 
\\$((\Pi x \in \A) \B(x))^{\perp\perp} = \bigcap_{ x \in \A} (x \R \B(x))^{\perp\perp}$.

\end{prop}

\begin{proof}
\begin{varitemize}
\item We prove first that $(\Pi x \in \A) \B(x) \subseteq  \bigcap_{ x \in \A} (x \R \B(x))^{\perp\perp}$.
\\ If $\de{D} \in (\Pi x \in \A) \B(x)$, then it is minimal s.t. for all $x \in\A$ $\llbracket \de{D}, x\rrbracket \in \B(x)$. Let $x_0 \in \A$, by definition of $x_0 \R \B(x_0)$ there exists $\de{D}'_{x_0} \subseteq \de{D}$ s.t. $\de{D}'_{x_0} \in x_0 \R \B(x_0)$. So $\de{D}'_{x_0} \in (x_0 \R \B(x_0))^{\perp\perp}$, $\de{D}'_{x_0} \subseteq \de{D}$ and then for all $x \in \A$ $\de{D} \in  (x \R \B(x))^{\perp\perp}$. Thus $(\Pi x \in \A) \B(x) \subseteq  \bigcap_{ x \in \A} (x \R \B(x))^{\perp\perp}$.
\item We remark that given a family of sets $(E_x)_{x \in N}$, $(\bigcap_{x \in N} E_x^{\perp\perp})^{\perp\perp}= \bigcap_{x \in N} E^{\perp\perp}_x$. Moreover given two sets of designs $E$ and $F$, if $E \subseteq F$ then $E^{\perp\perp} \subseteq F^{\perp\perp}$. 
\\Then from the first item we have $((\Pi x \in \A) \B(x))^{\perp\perp} \subseteq \bigcap_{ x \in \A} (x \R \B(x))^{\perp\perp}$. 
\item Let us prove that $\bigcap_{ x \in \A} (x \R \B(x))^{\perp\perp} \subseteq ((\Pi x \in \A) \B(x))^{\perp\perp}$.
\\Let $\de{D} \in \bigcap_{ x \in \A} (x \R \B(x))^{\perp\perp}$, then  for all $x \in \A$ $\de{D} \in (x \R \B(x))^{\perp\perp}$. Given $x_0 \in \A$, there exists a design $\de{D}_{x_0} \subseteq \de{D}$ s.t. $\de{D}_{x_0} \in | (x_0 \R \B(x_0))^{\perp\perp} |$. From Proposition \ref{prefreccia}, $x_0 \R \B(x_0)$ is principal, then $\de{D}_{x_0} \in (x_0 \R \B(x_0))^{\da}$. $\bigcup_{x \in \A}  \de{D}_{x} \subseteq \de{D}$ and $\llbracket \bigcup_{x \in \A}  \de{D}_{x} , x \rrbracket \in \B(x)$, then there exists a design $\de{F} \subseteq \bigcup_{x \in \A}  \de{D}_{x}$ s.t. $\de{F} \in ((\Pi x \in \A) \B(x))^{\da}$. Then $\de{F} \in ((\Pi x \in \A) \B(x))^{\perp\perp}$. $\de{F} \subseteq \de{D}$, therefore $\de{D} \in ((\Pi x \in \A) \B(x))^{\perp\perp}$. 
\\Thus $\bigcap_{ x \in \A} (x \R \B(x))^{\perp\perp} \subseteq ((\Pi x \in \A) \B(x))^{\perp\perp} $. 
\end{varitemize}
\end{proof}

\begin{rem}
The biorthogonal closure is really crucial to prove the previous proposition. Indeed given a principal set $\A$ and a family of principal sets $(\B(x))_{x \in \A}$, $(\Pi x \in \A) \B(x)$ is not always equal to $\bigcap_{x \in \A}  x \R \B(x)$. For instance let 
\\$\A= \{  \de{A}_1, \de{A}_2\}$, $\B(\de{A}_1)= \{\de{B}_1  \}$, $\B(\de{A}_2)= \{ \de{B}_2  \}$. 
\\\scalebox{.8}
{
$ \de{A}_1 = \shortstack{ $\hrulefill_{\emptyset}$ \\ $\vdash \alpha.01$  \\\hrulefill\\  $\alpha.0 \vdash $ \\\hrulefill\\ $\vdash \alpha$}$, \hspace{1em} $ \de{A}_2 = \shortstack{ $\hrulefill_{\emptyset}$ \\ $\vdash \alpha.1.2$  \\\hrulefill\\  $\alpha.1 \vdash $ \\\hrulefill\\ $\vdash \alpha$}$, \hspace{1em} $\de{B}_1 = \shortstack{ $\hrulefill_{\emptyset}$ \\$\vdash \beta$}$, \hspace{1em} $\de{B}_2 = \shortstack{ $\hrulefill_{\emptyset}$ \\ $\vdash \beta.0.2$  \\\hrulefill\\ $\beta.0 \vdash$ \\\hrulefill\\ $\vdash \beta$  }$,
}
\scalebox{.8}
{
$\de{D} = \shortstack{          \shortstack{     \shortstack{$\vdots$ \\ $\de{B}_1$ \\ $\vdash \beta$ \\\hrulefill\\ $\alpha.0.1 \vdash \beta$ \\\hrulefill\\ $\vdash \alpha.0, \beta$  }       \hspace{1em}    \shortstack{$\vdots$ \\ $\de{B}_2$ \\ $\vdash \beta$ \\\hrulefill\\  $\alpha.1.2 \vdash \beta$  \\\hrulefill\\  $\vdash \alpha.1, \beta$  }  }    \\\hrulefill\\   $\alpha\vdash \beta$} \quad $
}
\\The elements of $\de{A}_1 \R \B(\de{A_1})$ have first action $(-, \alpha , \{ 0\})$ while the elements of $ \de{A}_2 \R \B( \de{A}_2)$ have $(-, \alpha , \{ 1\})$, then $ \bigcap_{\de{A} \in \A} (\de{A} \R \B( \de{A} ) )$ is empty. $\de{D}$ belongs to $(\Pi x \in \A) \B(x)$, so $(\Pi x \in \A) \B(x)\neq \emptyset$. Thus $(\Pi x \in \A) \B(x) \neq \bigcap_{x \in \A}  x \R \B(x)$. 
\end{rem}

\begin{prop}\label{preprod}
Let $\A$ be a principal set and $(\B(x))_{x \in \A}$ a family of principal sets. Then $(\Pi x \in \A) \B(x)$ is principal.
\end{prop}

\begin{proof}
 We prove that $(\Pi x \in \A) \B(x)$ is $\da$-free and $((\Pi x \in \A) \B(x))^{\da} \subseteq  | ( (\Pi x \in \A) \B(x))^{\perp\perp}  | $ as in Proposition \ref{prefreccia}.
\\We just need to prove that $| ( (\Pi x \in \A) \B(x))^{\perp\perp}  | \subseteq ((\Pi x \in \A) \B(x))^{\da}$.
\\ From Proposition \ref{pinter} $((\Pi x \in \A) \B(x))^{\perp\perp} = \bigcap_{ x \in \A} (x \R \B(x))^{\perp\perp}$, so we want to show $   |  \bigcap_{ x \in \A} (x \R \B(x))^{\perp\perp}  |  \subseteq  ((\Pi x \in \A) \B(x))^{\da} $.\\
 Let $\de{D} \in |  \bigcap_{ x \in \A} ( x \R \B(x))^{\perp\perp}|$. We want to prove that $\de{D} \in (( \Pi x \in \A )  \B(x) )^{\da}$, i.e., $\de{D}$ is minimal s.t. for all $x \in \A$, $\llbracket \de{D}, x \rrbracket \in (\B(x))^{\da}$. Given $x \in \A$, there exists $\de{D}_{x} \subseteq \de{D}$ s.t. $\de{D}_{x} \in   |   (x \R \B(x))^{\perp\perp} | = (x \R \B(x))^{\da}$.  Then for all $\de{A} \in \A$ $ \llbracket \bigcup_{x \in \A} \de{D}_{x} , \de{A} \rrbracket \in (\de{A} \R \B(\de{A}))^{\da}$. $\bigcup_{x \in \A} \de{D}_{x} \subseteq \de{D}$ and  there exists $\de{L} \subseteq \bigcup_{x \in \A} \de{D}_{x} $ s.t. $ \de{L} \in (( \Pi x \in \A )  \B(x) )^{\da}$.\\
 We prove by contradiction that $\de{D} \in  (( \Pi x \in \A )  \B(x) )^{\da}$.\\
 Let $\de{D} \notin (( \Pi x \in \A )  \B(x) )^{\da}$, then $\de{L} \subsetneq \de{D}$. Moreover $\de{L} \in  (( \Pi x \in \A )  \B(x) )^{\da}$ implies that $\de{L} \in ( (\Pi x \in \A) \B(x))^{\perp\perp}$. Then $\de{D} \notin |( (\Pi x \in \A) \B(x) )^{\perp\perp} |$ (contradiction).\\
So $\de{D} \in (( \Pi x \in \A )  \B(x) )^{\da}$.
\\Thus $ | ( (\Pi x \in \A) \B(x))^{\perp\perp}  | \subseteq ((\Pi x \in \A) \B(x))^{\da}$. 
\end{proof}
We sum up below the interpretation of the elimination and the equality rule for $(\Pi x \in A) B(x)$.
\begin{center}
\begin{tabular}{| l | l | } \hline
{\textbf{Martin-L\"of }$\Pi$-\textbf{rules}} &  {\textbf{Ludics}} \\ \hline 
\hspace{1em} $(\la x) b(x) \in (\Pi x \in A) B(x)$ & \hspace{1em}$ \de{D} \in ( \Pi x \in \A )  \B(x) $\\
\hspace{1em}$c \in (\Pi x \in A) B(x)$ & \hspace{1em} $\de{R}$ cut-net s.t. $\llbracket \de{R}\rrbracket \in (\Pi x \in \A) \B(x)$ \\
\hspace{1em}$Ap(c,a)$ where $a \in A$ & \hspace{1em} $\llbracket \de{R}, \de{A} \rrbracket $ where $\de{A} \in \A$\\
\hspace{1em}$Ap((\la x ) b(x), a)= b(a) \in B(a)$ & \hspace{1em} $\llbracket \de{D}, \de{A} \rrbracket = \llbracket \llbracket \de{R} \rrbracket , \de{A} \rrbracket= \llbracket \de{R}, \de{A} \rrbracket \in \B(\de{A}) $ \\ \hline
\end{tabular}
\end{center}

\begin{example}
We give two examples of elements of $(\Pi \textbf{n} \in \N) \Ln$.\\
As a first example, let $p = <p_{i}>_{i\geq 1}$ be an infinite sequence of integers represented by designs $(\mathfrak p_{i})_{i\geq 1}$ of $\N$ of base $\vdash \xi$, the design $\de{E}_{\sigma\vdash\xi}^{p}$ defined below belongs to $(\Pi \textbf{n} \in \N) \Ln$. This design builds for each $n \in \mathbb{N}$ the design representing the list $<p_{1}, \dots, p_n>$ of length $n$. 
\\ For all $i \geq 0$, \scalebox{.9}{ $\de{E}^{p}_{\sigma.\overline{i} \vdash \xi}     = \shortstack{   \shortstack{$\quad \quad \de{D}^{<p_{1},...,p_{i}>}_{\xi}$ \\ $\vdash \xi$} \hspace{1em}   \shortstack{$ \de{E}^{p}_{\sigma.\overline{i+1} \vdash \xi}$ \\\hrulefill\\ $\vdash \sigma.\overline{i}.0, \xi$} \\\hrulefill \\ $\sigma.\overline{i} \vdash \xi $     }$}, where $\de{D}^{<p_{1},...,p_{i}>}_{\xi}$ represents the list $<p_{1},...,p_{i}>$ on the base $\vdash \xi$. When $\de{E}_{\sigma\vdash\xi}^{p}$ interacts with $\textbf{n} \in \N$, it reads all the actions of $\textbf{n}$ and then gives the design which represents $<p_{1},...,p_{n}>$. We define another design, $\de{G}^{l}_{\sigma \vdash \xi}$, which gives lists that do not always have the same prefixes. Let $l=\{ l_{i} \, | \, i \in \mathbb{N}\}$, where $l_{i}$ is a list of length $i$. 
\\For all $i \geq 1$, \scalebox{.9}{$\de{G}^{l}_{\sigma.\overline{i} \vdash \xi} =  \shortstack{  \shortstack{$ \de{D}^{l_{i}}_{\xi} $ \\ $\vdash \xi$}  \hspace{1em}  \shortstack{  $\de{G}^{l}_{\sigma.\overline{i+1} \vdash \xi}$   \\\hrulefill\\ $\vdash \sigma.\overline{i}.0, \xi $ }      \\\hrulefill\\ $\sigma.\overline{i} \vdash \xi $ }$}, where $ \de{D}^{l_{i}}_{\xi} $ represents the list $l_{i}$. 
\end{example}

\subsection{The type $(\Sigma x \in A) B(x)$}  \label{sigma}
Given a type $A$ and a family of types $(B(x))_{x \in A}$ respectively represented by means of the principal set $\A$ based on $\vdash \sigma.1.1$ and the family of principal sets $(\B(x))_{x \in \A}$ based on $\vdash \sigma.2.2$, we represent the canonical terms  of type $(\Sigma x \in A) B(x)$ by
\\ $ (\Sigma x \in \A) \B(x) : = \{   \de{D}_{\de{A}, \de{B}}  \, | \, \de{A} \in \A, \de{B} \in \B(\de{A})       \}$ where $\de{D}_{\de{A}, \de{B}}= \scalebox{.7} {$\shortstack{         \shortstack{   \shortstack{     $\de{A}$ \\ $\vdash \sigma.1.1$  \\\hrulefill\\  $ \sigma.1 \vdash$}        \hspace{1em}    \shortstack{     $\de{B}$ \\ $\vdash \sigma.2.2$  \\\hrulefill\\  $\sigma.2 \vdash$}     }    \\\hrulefill\\ $\vdash \sigma$ }$}$.
 \\Remark that  $\de{D}_{\de{A}, \de{B}}$ may be seen as a tensor product $\odot$ defined in \cite{LocSol}. 
 \\The elements of $\Sig$ are $ (+, \sigma, \{1,2\}) (-, \sigma.1, \{1\}) \de{A} \cup (-, \sigma.2, \{2\}) \de{B}$, for all $\de{A} \in \A$ and $\de{B} \in \B$, that we denote as $w \de{A} \odot w' \de{B}$ where $ w = (+,\sigma, \{ 1  \}) (-, \sigma.1, \{ 1\})$ and $ w' =  (+,\sigma, \{ 2  \}) (-, \sigma.2, \{ 2\})$.

   \begin{prop}
   
   Given a principal set $\A$ and a family of principal sets $(\B(x))_{x \in \A}$,  $\Sig$ is principal.   
   \end{prop}
 
 \begin{proof}
$\Sig$ is $\da$-free because $\A$, $\B$, $w$ and $w'$ are $\da$-free.
\\Now we prove $|(\Sig )^{\perp\perp}|= (\Sig )^{\da}  $. 
\\Remark that $\Sig = \bigcup_{\de{A} \in \A, \de{B} \in \B(\de{A})} \{ w \de{A} \odot w' \de{B}\} $. From \cite{LocSol} given two sets  of designs $E,F $ we have $|( E \odot F)^{\perp\perp}| = |E^{\perp\perp}| \odot  |F^{\perp\perp}| $. 
Then we get $  \bigcup_{\de{A} \in \A, \de{B} \in \B(\de{A})}    |  (  w \de{A} \odot  w' \de{B} )^{\perp\perp}     | =  \bigcup_{\de{A} \in \A , \de{B} \in \B(\de{A})}  | ( w \de{A})^{\perp\perp} | \odot  | ( w' \de{B})^{\perp\perp}  |   $. \\
For all $\de{A} \in \A$ and $\de{B} \in \B(\de{A})$, $w \de{A}$ and $w' \de{B}$ are principal, then $ | ( w \de{A})^{\perp\perp} | \odot  | ( w' \de{B})^{\perp\perp}  |  =  ( w \de{A})^{\da}$$ \odot ( w' \de{B})^{\da} $. 
\\Thus $| (\Sig)^{\perp\perp} |= (\Sig)^{\da}$.  
\end{proof}
So $\Sig$ represents the canonical terms of type $(\Sigma x \in A)B(x)$.

 \begin{example}
Let 
\scalebox{.9}
{$\pi_{1} = \shortstack{  $Id$ \\ $\sigma.1.1 \vdash \alpha$ \\\hrulefill\\ $\vdash \sigma.1, \sigma.2 , \alpha$  \\\hrulefill\\ $\sigma \vdash \alpha$}$, \hspace{2em} $\pi_{2} = \shortstack{  $Id$ \\ $\sigma.2.2 \vdash \alpha$ \\\hrulefill\\ $\vdash \sigma.1, \sigma.2 , \alpha$  \\\hrulefill\\ $\sigma \vdash \alpha$}$.} 
\\Then for all $\de{D}_{\de{A}, \de{B}} \in \Sig$, $\llbracket \pi_{1}, \de{D}_{\de{A}, \de{B}} \rrbracket= \de{A}$ and  $\llbracket \pi_{2}, \de{D}_{\de{A}, \de{B}}  \rrbracket = \de{B}$, i.e., $\pi_{1}$ and $\pi_{2}$ respectively represent the projection on the first and the second component. 
 \end{example}

 We define the function sum as follows. 
 \begin{example}
The sum.\\
 We represent a pair of natural numbers $(n,m)$ by 
$\de{E}_{(n,m)}= \quad \shortstack{ \shortstack{$\textbf{n}_{\sigma.1.1}$ \\\hrulefill\\ $\sigma.1 \vdash$} \hspace{1em} \shortstack{$\textbf{m}_{\sigma.2.2}$ \\\hrulefill\\ $\sigma.2 \vdash$} \shortstack{} \\\hrulefill\\$\vdash \sigma$}$. The function sum is then represented by $\de{S}_{+}$ 

\begin{center}
$\de{S}_{+} =      \shortstack{     $\de{G}_{0}$  \\  $\vdash \sigma.1, \sigma.2, \alpha$     \\\hrulefill\\ $\sigma \vdash \alpha$}  $, where $\de{G}_{i} = \shortstack{            \shortstack{ $\de{F}_{0}$\\ $\vdash \sigma.2, \alpha.\overline{i}$  }          \hspace{1em}      \shortstack{  $\de{G}_{i+1}$  \\$\vdash  \sigma.1.1 \overline{i} 0, \sigma.2, \alpha. \overline{i+1}$  \\\hrulefill\\\hrulefill\\ $\vdash \sigma.1.1 \overline{i} 0, \sigma.2, \alpha. \overline{i}$  }  \\\hrulefill\\  $\sigma.1.1 \overline{i} \vdash \alpha. \overline{i}, \sigma.2$  }$ and $\de{F}_{i}=   \shortstack{ \shortstack{  $\hrulefill_{\emptyset}$ \\ $\vdash \alpha.  \overline{i}$ }       \hspace{1em}   \shortstack{ $\de{F}_{i+1}$ \\ $\vdash \sigma.2.2  \overline{i} 0 , \alpha.  \overline{i+1}$  \\\hrulefill\\\hrulefill\\ $ \vdash \sigma.2.2  \overline{i} 0 , \alpha.  \overline{i}  $ }   \\\hrulefill\\ $\sigma.2.2 \overline{i} \vdash \alpha.  \overline{i}$  }$ $\forall i \in \mathbb{N}$. 
\end{center}
$\llbracket   \de{S}_+  , \de{E}_{(n,m)}   \rrbracket = \textbf{n\pmb{+}m}$. The intuition behind $\de{S}_{+}$ is the following: read $n$ and stock step by step $\textbf{n}$ on $\alpha$ and then do the same with $\textbf{m}$. In particular the design $\de{G}_i$ reads $n$, while $\de{F}_i$ reads $m$.  

\end{example}

 \subsection*{Conclusion and future work}
In this paper, we propose a representation of Martin-L\"of types in Ludics. We define for some simple types, dependent product type and $\Sigma$ type a set of designs representing their canonical terms and we proved that such a set is principal. Since Ludics is affine, this framework is quite restricted. However this is not a real problem as  we may apply our approach on extensions of Ludics that integrate exponentials \cite{Ludicswithrep}. We intend also to work with Computational Ludics, i.e., a reformulation of Ludics from a computational point of view, introduced by Terui \cite{TeruiComp}, which overcomes the linear framework. We proposed a representation of records in Computational Ludics (master thesis), then the continuation of this work would be to make explicit links between Computational Ludics and Type Theory with records \cite{Cooper}, concerning their relevance for Linguistics application. 
   
   \bibliography{biblio} 
\bibliographystyle{plain}

\end{document}